\documentclass[12pt]{amsart}
\usepackage{amssymb,latexsym, amsmath, amsxtra}
\usepackage{graphicx}
\newtheorem{theorem}{Theorem}
\newtheorem{lemma}{Lemma}
\newtheorem{corollary}{Corollary}

\newtheorem{conjecture}{Conjecture}
\newtheorem{remark}{Remark}
\begin{document}

\title{Correlations of eigenvalues and Riemann zeros}
\author {J. B. Conrey \and N. C. Snaith}

 \maketitle

\tableofcontents

\section{Introduction}

In 1972 Montgomery \cite{kn:mont73} and Dyson \cite{kn:dyson}
discovered that pairs of zeros of the Riemann zeta-function are
distributed like pairs of eigenvalues of random unitary matrices.
Part of this discovery could be proven, under the assumption of
the Riemann Hypothesis, and part relies on a heuristic based on
the Hardy-Littlewood conjectures for the distribution of prime
pairs.

Odlyzko \cite{kn:odlyzko89}, in the 1980s carried out a
substantial numerical test of Montgomery's conjecture which
provided stunning visual substantiation.

Subsequently Rudnick and Sarnak \cite{kn:rudsar} showed that the
limit high on the critical line of the $n$-correlation of the
zeros of the Riemann zeta-function agreed with that of unitary
matrices, provided that the test function had a Fourier transform
with limited support.

At the same time, Bogomolny and Keating \cite{kn:bogkea96} showed
how the Hardy-Littlewood conjectures could be used to derive the
asymptotic limit of the $n$-correlation.

Bogomolny and Keating \cite{kn:bk96} also investigated the
difference between Montgomery's limiting pair-correlation
conjecture and the data of Odlyzko. A close examination of
Odlyzko's data revealed that lower order terms, likely to be of an
arithmetic nature, were present. They derived formulae for these
lower order terms, initially using the Hardy-Littlewood
conjectures, but subsequently developing a method whose point of
departure was the trace formula of Gutzwiller. They gave full
details for the lower order terms for the 2-point correlation, as
well as numerics showing the goodness of fit, whereas for three
and higher correlations, they outlined several methods which lead
to these lower order terms.

In this paper, we present a different approach to obtaining these
lower order terms for $n$-correlation. Our approach is based on
the `ratios conjecture' of Conrey, Farmer, and Zirnbauer
\cite{kn:cfz1,kn:cfz2} (see also \cite{kn:consna06}). Assuming the
ratios conjecture we prove a formula which explicitly gives all of
the lower order terms in any order correlation. (In the final
section we write down the first four correlations.)

Our method works equally well for random matrix theory. An
interesting feature of this work is the new formula for the
$n$-correlation of random matrix theory that arises by this method
(see Theorem \ref{theo:main}.)  It is a far less elegant formula
than the usual determinantal expression, but it allows for direct
comparison with the number theoretical result, illustrating the
identical structure of the $n$-point correlations of Riemann zeros
and random matrix eigenvalues. In fact, in the scaling limit (when
the variables in the test function are multiplied by $\log T/2\pi$
and $T$, the height up the critical line, becomes large) then all
of the arithmetic features of the formula for the $n$-correlation
of the Riemann zeros disappear and it exactly matches our new
formula for the the $n$-correlation of eigenvalues of unitary
matrices in the equivalent limit. This identification allows us to
prove that in the scaling limit the leading order terms for our
$n$-correlation of the Riemann zeros have the expected
determinantal form.  See \cite{kn:consna07} for the explicit
derivation of the asymptotic limit in the case of the triple
correlation of Riemann zeros.  The higher correlations follow in
exactly the same way.

  This point is significant in view of the difficulty
  in making this identification in other works on $n$-correlation
  and $n$-level density.
  In Rudnick and Sarnak this identification is proven in the
  case of test functions restricted to $[-1,1]$; the proof is quite involved and
  makes serious use of the restriction on the support of the test
  function.
Indeed this point forms a difficulty which shows up, for example,
in the work of Gao \cite{kn:gao05} on $n$-level density for zeros
of quadratic $L$-functions. Rubinstein \cite{kn:rub98} had
evaluated this for test functions whose total support was
contained in $[-1,1]$ and verified the determinantal form for
functions restricted to this class, analogous to what Rudnick and
Sarnak did. Gao extended the range of support to $[-2,2]$ but for
these test functions was unable to derive the determinantal form,
due to combinatorial complexities. Here we handle the case in full
generality without any mention of the test function. It is
possible that our method will shed light on this difficulty that
arises in these other works.

This paper extends the calculation of the triple correlation of
Riemann zeros \cite{kn:consna07}.  An anticipated application of
this current work is to the determination of the lower order terms
in the nearest neighbor spacing for zeta-zeros.

Throughout this paper we assume the truth of the Riemann
Hypothesis.

\section{Background and notation}

\subsection{The Riemann zeta-function}
The Riemann zeta-function is defined by
\begin{eqnarray}\zeta(s)=\sum_{n=1}^\infty \frac{1}{n^s}
\end{eqnarray} for $s=\sigma+it$ with $\sigma >1.$ It has a
meromorphic continuation to the whole complex plane with its only
singularity a simple pole at $s=1$ with residue 1.  It satisfies a
functional equation which, in its symmetric form reads
\begin{eqnarray}
\pi^{-\frac s 2}\Gamma\bigg(\frac s 2 \bigg) \zeta(s) = \pi^{
\frac {s-1} 2}\Gamma\bigg(\frac {1-s} 2 \bigg) \zeta(1-s)
\end{eqnarray}
and in its asymmetric form is
\begin{eqnarray}
\zeta(s)=\chi(s)\zeta(1-s)\end{eqnarray} where
\begin{eqnarray}
\chi(1-s)=\chi(s)^{-1}=2(2\pi)^{-s}\Gamma(s)\cos \frac{\pi s}{2}.
\end{eqnarray}
The product formula discovered by Euler is
\begin{eqnarray}
\zeta(s)=\prod_p \bigg( 1-\frac{1}{p^s}\bigg)^{-1}
\end{eqnarray}
for $\sigma>1$ where the product is over the prime numbers $p$.

The complex zeros of the Riemann zeta-function are denoted by
$\rho=\beta+i\gamma_j$ where it is known that $0<\beta<1.$
  The Riemann Hypothesis asserts that $\beta=1/2$ for all zeros $\rho$. We assume
this is true and denote the zeros as $1/2+i\gamma_j$, where
$0<\gamma_1\le \gamma_2\le \dots.$ The number of $\gamma$ with
$0<\gamma\le T$ is given by
\begin{eqnarray}
N(T)=\#\{\gamma\le T\}=\frac{T}{2\pi}\log \frac{T}{2\pi e} +O(\log
T)
\end{eqnarray}
so that the average distance from one $\gamma$ to the next is
$\sim2\pi/\log \gamma$.

The family $\{\zeta(1/2+it)| t>0\}$ parametrized by real numbers
$t$ can be modeled by characteristic polynomials of unitary
matrices.

\subsection{Unitary matrices}
If $X$ is an $N\times N$ matrix with complex entries $X=(x_{jk})$,
we let $X^*$ be its conjugate transpose, i.e. $X^*=(y_{jk})$ where
$y_{jk}=\overline{x_{kj}}.$ $X$ is said to be unitary if $XX^*=I$.
We let $U(N)$ denote the group of all $N\times N$ unitary
matrices. This is a compact Lie group and has a Haar measure which
allows us to do analysis.

All of the eigenvalues of $X\in U(N)$ have absolute value 1; we
write them as
  \begin{eqnarray}e^{i\theta_1}, e^{i\theta_2}, \dots , e^{i\theta_N}.\end{eqnarray}
The eigenvalues of $X^*$ are $e^{-i\theta_1},\dots,
e^{-i\theta_N}$. Clearly, the determinant, $\det X=\prod_{n=1}^N
e^{i\theta_n}$ of
  a unitary matrix is a complex number
with absolute value equal to 1.

The average distance from one $\theta$ to the next is $2\pi/N$. To
obtain a sequence of numbers with average spacing 1 we let
\begin{eqnarray}
\tilde{\theta_j}=\frac{N\theta_j}{2\pi}.
\end{eqnarray}

For any sequence of $N$ points on the unit circle there are
matrices in $U(N)$ with these points as eigenvalues. The
collection of all matrices with the same set of eigenvalues
constitutes a conjugacy class in $U(N)$.  Thus, the set of
conjugacy classes can identified with the collection of sequences
of $N$ points on the unit circle.

We are interested in computing various statistics about these
eigenvalues. Consequently, we identify all matrices in $U(N)$ that
have the same set of eigenvalues.
  Weyl's integration formula gives a simple way to perform averages over $U(N)$
  for functions $f$ that are constant on conjugacy classes.
  Such functions are called `class functions'.
Weyl's formula asserts that for such an $f$,
\begin{eqnarray}\int_{U(N)} f(X) ~d\mbox{Haar}=\int_{[0,2\pi]^N}
f(\theta_1,\dots,\theta_N)dX_N,\end{eqnarray} where
  \begin{eqnarray} dX_N &=&\prod_{1\le j<k\le
N}\big|e^{i\theta_k}-e^{i\theta_j}\big|^2 ~\frac{d\theta_1 \dots
d\theta_N}{N! (2\pi)^N}.\end{eqnarray} Since $N$ will be fixed in
this paper, we will usually write $dX$ in place of $dX_N$. The Haar
measure can be expressed in terms of the Vandermonde determinant
\begin{eqnarray}
\Delta(w_1,\dots,w_R)=\det_{R\times R}\big(
w_i^{j-1}\big)=\prod_{1\le j < k\le R}(w_k-w_j).
\end{eqnarray}

The characteristic polynomial of a matrix $X$ is denoted
$\Lambda_X(s)$ and is defined by
\begin{eqnarray}\Lambda_X(s)=\det(I-sX^*)=\prod_{n=1}^N(1-se^{-i\theta_n}).
\end{eqnarray}
The roots of $\Lambda_X(s)$ are the eigenvalues of $X$ and are on
the unit circle. The characteristic polynomial  satisfies the
functional equation
\begin{eqnarray} \Lambda_X(s)&=&(-s)^N\prod_{n=1}^N e^{-i\theta_n}\prod_{n=1}^N
(1-e^{i\theta_n}/s)\\
&=&(-1)^N \det X^* ~s^N~\Lambda_{X^*}(1/s).\end{eqnarray} Note that
\begin{eqnarray} \label{eqn:fe}
s\frac{\Lambda_X'}{\Lambda_X}(s)+\frac 1 s
\frac{\Lambda_{X^*}'}{\Lambda_{X^*}}\big(\frac 1s\big)=N.
\end{eqnarray}
These characteristic polynomials have value distributions similar
to that of the Riemann zeta-function and form the basis of Random
Matrix models which predict behavior for the Riemann zeta-function
based on what can be proven about the $\Lambda$. Some care has to
be taken in making these comparisons because we are used to
thinking about the zeta-function in a half-plane whereas the
$\Lambda$ are naturally studied in a circle. The translation is
that the 1/2-line corresponds to the unit circle; the half-plane
to the right of the 1/2-line corresponds to the inside of the unit
circle. Note that $\Lambda_X(0)=1$ is the analogue of
$\lim_{\sigma\to \infty}\zeta(\sigma+it)=1$.

We let
\begin{eqnarray}\label{eq:z}
z(x)=\frac{1}{1-e^{-x}}.
\end{eqnarray}
In our formulas for averages of characteristic polynomials the
function $z(x)$ plays the role for random matrix theory that
$\zeta(1+x)$ plays in the theory of moments of the Riemann
zeta-function.

We want  an accurate formula for
\begin{eqnarray}
\sideset{}{^*}\sum_{0<\gamma_{j_1},\dots,\gamma_{j_n}<T}f(\gamma_{j_1},\dots,\gamma_{j_n}),
\end{eqnarray}
for suitable functions $f$, to be described later, where the sum
is for distinct indices $j$; the desired formula should be
analogous to the RMT theorem which we state in the next section.
\section{Eigenvalue correlations }
Here is a statement for $n$-correlation of eigenvalues of random
unitary matrices of size $N$:
\begin{theorem}
Let $f:[0,2\pi]^n\to \mathbb C$ be a continuous function of
$n$-variables. Then
\begin{eqnarray*}
\int_{U(N)}\sideset{}{^*}\sum_{1\le j_1,\dots , j_n\le N}
f(\theta_{j_1},\dots ,\theta_{j_n}) dX_N = \frac{1}{(2\pi)^n}
\int_{[0,2\pi]^n} f(\theta_1,\dots,\theta_n)
  \det_{n\times n} S_N(\theta_k-\theta_j)~d\theta_1\dots
  ~d\theta_n,
\end{eqnarray*}
where $\sideset{}{^*}\sum$ indicates that the sum is for distinct
indices and where
\begin{eqnarray*}
S_N(\theta)=\frac{\sin \frac {N\theta}{2}}{\sin \frac \theta 2 }.
\end{eqnarray*}
\end{theorem}
This theorem is a well-known consequence of Gaudin's Lemma, see
\cite{kn:conrey04}.

In the following sections we present a new proof of this theorem,
for periodic, holomorphic test functions $f$, based on a formula
for averaging ratios of characteristic polynomials of unitary
matrices.  There are many proofs of this formula, see
\cite{kn:cfz1,kn:cfs05,kn:bumgam06}. The point of this approach is
that it has a natural analog in the theory of $L$-functions.

\subsection{Averages of ratios of characteristic polynomials}
\label{sect:ratiolambda}

The statement of the ratios theorem is slightly complicated. We
attempt to make it easier to comprehend by eliminating subscripts.
So, let there be given finite sets $A, B, C$ and $D$ and consider
\begin{eqnarray} \label{eq:R}
&&\mathcal R(A,B;C,D):=\int_{U(N)}\frac{\prod_{\alpha\in A}
\Lambda_X(e^{-\alpha})\prod_{\beta\in B}
\Lambda_{X^*}(e^{-\beta})} {\prod_{\gamma\in
C}\Lambda_X(e^{-\gamma}) \prod_{\delta\in D}
\Lambda_{X^*}(e^{-\delta})}dX,
\end{eqnarray}
with $\Re \gamma>0, \Re \delta >0$.  Theorem \ref{theo:rat}, the
ratios theorem, is written in an equivalent but slightly different
form to previous work, where we express $\mathcal R(A,B;C,D)$ as a
sum over subsets $S\subset A$ and $T\subset B$ with $|S|=|T|$. Each
term in this sum essentially has the same structure, except that the
elements of $S$ effectively exchange places with those in $T$. In
addition we let $\overline{S}=A-S$ and $\overline{T}=B-T$. We will
let $\hat \alpha$ denote a generic member of $S$ and $\hat \beta$
denote a generic member of $T$; we will use $\alpha$ and $\beta$ for
generic members of $A$ and $B$ or of $\overline{S}$ and
$\overline{T}$, according to the context. Also $S^-=\{-\hat\alpha:
\hat\alpha\in S\}$, and similarly for $T^-$. The Ratios Theorem is
most easily stated in terms of
\begin{eqnarray}
Z(A,B):=\prod_{\alpha\in A\atop\beta\in B}z(\alpha+\beta),
\end{eqnarray}
where $z(x)=\frac{1}{1-e^{-x}}$, and
\begin{eqnarray}\label{eq:Z}Z(A,B;C,D):=
\frac{\prod_{\alpha\in A\atop \beta\in B}
z(\alpha+\beta)\prod_{\gamma\in C\atop \delta\in
D}z(\gamma+\delta)} {\prod_{\alpha\in A\atop  \delta\in D}
z(\alpha+\delta) \prod_{\beta\in B\atop \gamma\in C}
z(\beta+\gamma)}=\frac{Z(A,B)Z(C,D)}{Z(A,D)Z(B,C)}.\end{eqnarray}
\begin{theorem}[Ratios Theorem \cite{kn:cfz1,kn:cfs05}] \label{theo:rat}
With $\Re \gamma>0, \Re \delta >0$ for $\gamma\in C$ and $\delta
\in D$, $|C|\leq|A|+N$ and $|D|\leq|B|+N$, we have
\begin{eqnarray*}
&&\mathcal R(A,B;C,D)=\sum_{S\subset A,T\subset B\atop
|S|=|T|}e^{-N(\sum_{\hat\alpha\in S} \hat \alpha +\sum_{\hat\beta
\in T}\hat\beta)} Z(\overline{S}+ T^-,\overline{T}+ S^-;C,D),
\end{eqnarray*}
where $A=S+\overline{S}$, $B=T+\overline{T}$ and $Z$ is defined at
(\ref{eq:Z}).
\end{theorem}

\subsection{Averages of logarithmic derivatives of characteristic
polynomials}

For use in determining multiple correlation we differentiate the
Ratios Theorem to obtain a theorem about averages of logarithmic
derivatives:

\begin{theorem}
\label{theo:J} If $\Re \alpha_j>0$ and $ \Re \beta_j
>0$ for  $\alpha_j\in A$ and $\beta_j\in B$, then $J(A;B)=J^*(A;B)$  where
\begin{eqnarray}
J(A;B)&:=& \int_{U(N)}\prod_{\alpha\in A}
(-e^{-\alpha})\frac{\Lambda_X'}{\Lambda_X}(e^{-\alpha})\prod_{\beta\in
B} (-e^{-\beta})\frac{\Lambda_{X^*}'}{\Lambda_{X^*}}(e^{-\beta})~ dX
,
\end{eqnarray}
 \begin{eqnarray} &&J^*(A;B):= \nonumber \\
 &&\qquad\qquad\sum_{S\subset A,T\subset B\atop
|S|=|T|}e^{-N(\sum_{\hat \alpha\in S} \hat \alpha
+\sum_{\hat{\beta}\in T}\hat\beta)} \frac{Z(S,T)Z(S^-,T^-)} {
Z^{\dagger}(S,S^-)Z^{\dagger}(T,T^-)} \sum_{{(A-S)+ (B-T)\atop =
U_1+\dots + U_R}\atop |U_r|\le 2}\prod_{r=1}^R H_{S,T}(U_r),
\end{eqnarray}
and
\begin{equation}\label{eqn:Hrmt}
H_{S,T}(W)=\left\{\begin{array}{ll} \sum_{\hat \alpha\in
S}\frac{z'}{z}(\alpha-\hat{\alpha})-\sum_{\hat\beta\in T}
\frac{z'}{z}(\alpha +\hat \beta) &\mbox{ if $W=\{\alpha\}\subset
A-S$}
   \\
\sum_{\hat\beta\in T}\frac{z'}{z}(\beta-\hat
\beta)-\sum_{\hat\alpha\in S} \frac{z'}{z}
(\beta+\hat\alpha) &\mbox{ if  $W=\{\beta\}\subset B-T$}\\
\left(\frac{z'}{z}\right)'(\alpha+\beta) & \mbox{ if
$W=\{\alpha,\beta\}$ with $
{\alpha \in A-S, \atop \beta\in B-T}$}\\
0&\mbox{ otherwise}.
\end{array}\right.
\end{equation}
Also, $Z(A,B)=\prod_{\alpha\in A\atop\beta\in B}z(\alpha+\beta)$,
with the dagger on $Z^\dagger(S,S^-)$ imposing the additional
restriction that a factor $z(x)$ is omitted if its argument is
zero.

\end{theorem}

\begin{remark}The definitions of $J(A;B)$ and $J^*(A;B)$  make sense
without the restriction that $\Re \alpha_j>0,$ and $ \Re \beta_j
>0$. However, the two are not equal without these
restrictions.\end{remark}

\begin{remark}Note that $J^*(A;B)$  has a pole when an $\alpha\in A$ is
equal to $-\beta$, for some $\beta \in B$. It also appears to have
a pole when two $\alpha$'s are equal, say $\alpha_1=\alpha_2$,
occurring when $\alpha_1\in S$ and $\alpha_2\notin S$, as seen in
the term $\frac{z'}{z}(\alpha-\hat\alpha)$ of (\ref{eqn:Ha}).
However, this is cancelled by a pole with residue of the opposite
sign when $S$ is replaced by $S-\{\alpha_1\}+ \{\alpha_2\}$. The
same phenomenon occurs when two $\beta$'s are equal, as can be
seen in the concrete examples given in (\ref{eqn:Jabb}) and
(\ref{eqn:Jabbb}).\end{remark}

\begin{proof} By (\ref{eq:R}), we have
\begin{eqnarray} \label{eq:Jasderiv}
J(A;B) =\left.\prod_{\alpha\in A \atop \beta \in
B}\frac{d}{d\alpha}\frac{d}{d\beta} \mathcal
R(A,B;C,D)\right|_{C=A \atop D=B}.
\end{eqnarray}
Of course, in this situation $|C|=|A|$ and $|D|=|B|$; so that we
may think of $A=\{\alpha_1,\dots ,\alpha_k\}$ and
$C=\{\gamma_1,\dots,\gamma_k\}$ and then the substitution
``$C=A$''means the substitution $\gamma_i=\alpha_i$ for
$i=1,2,\dots, k$, and similarly for $D$ and $B$.

Recall from Theorem \ref{theo:rat} that $\mathcal  R$ is expressed
as a sum of $Z$ over subsets $S$ and $T$.  In performing the
differentiations in (\ref{eq:Jasderiv}) we will find that the
derivatives with respect to the variables in $S$ and $T$ are fairly
simple to perform (as we will show below, culminating in
(\ref{eq:JSTderiv})), but we will need Lemma \ref{lemma:diff} to
differentiate with respect to the remaining variables.  Hence we
first
 rewrite $Z$ so as to separate these variable types.

 Note first that $Z(A,B)=Z(B,A)$ and that
$Z$ behaves nicely with respect to unions:
\begin{eqnarray}\label{eqn:Zunion}
Z(A_1+ A_2,B)=Z(A_1,B)Z(A_2,B).
\end{eqnarray}
Recall that $A=S+ \overline{S}$ and $B=T+ \overline{T}$ and put
$C=C_S+ C_{\overline{S}}$ and $D=D_T+ D_{\overline{T}}$ where we
think of $C_S$, for example, as being the set that will be
substituted by $S$ when eventually $C$ is substituted by $A$. Then
using (\ref{eqn:Zunion}) repeatedly,  we have
\begin{eqnarray}\label{eq:Zswitched}
&&Z(\overline{S}+ T^-,\overline{T}+ S^-;C,D)\\\nonumber
&&\qquad\qquad=\frac{Z(\overline{S},\overline{T})Z(\overline{S},S^-)Z(T^-,\overline{T})Z(T^-,S^-)Z(C,D)}{Z(\overline{S},D)Z(\overline{T},C)Z(S^-,C_{\overline{S}})Z(T^-,D_{\overline{T}})Z(S^-,C_S)Z(T^-,D_T)}.
\end{eqnarray}
This simplifies further if we make the substitution for $C$ and $D$:
\begin{eqnarray} \label{eqn:nonsense}
  Z(\overline{S}+ T^-,\overline{T}+
S^-;C,D)\big|_{C=A\atop
D=B}=\frac{Z(S,T)Z(S^-,T^-)}{Z(S,S^-)Z(T,T^-)}=Z(S,T;T^-,S^-).
\end{eqnarray}
Note that since  $z(x)$ has a pole at $x=0$, the resulting
expression is 0 unless both $S$ and $T$ are empty.

    We now differentiate (\ref{eq:Zswitched}) with respect
to the variables in $S$ and $T$; these derivatives are easy to
calculate because, anticipating the substitution of each
$\gamma\in C_S$ by an $\hat \alpha$ we see that in differentiating
with respect to $\hat \alpha$ the expression $z(\gamma-\hat
\alpha)$ in the denominator
 (one of the factors of
$Z(S^-,C_S)$) must be differentiated; if not it makes the whole
expression 0 after the substitution is made because $z(x)$ has a
pole at $x=0$. Using the notation
\begin{eqnarray}
Z^{\dagger}(X,Y)=\prod_{x\in X,y\in Y\atop x+y\ne 0}z(x+y),
\end{eqnarray}
and noting that
\begin{eqnarray} \frac{d}{d\hat\alpha}
\frac{1}{z(\gamma-\hat\alpha)} = -e^{\hat\alpha-\gamma},
\end{eqnarray}
we have, for example,
\begin{eqnarray}
\prod_{\hat\alpha\in S}\frac{d}{d\hat\alpha}
\frac{1}{Z(S^-,C_S)}\bigg|_{C_S=S}=\frac{(-1)^{|S|}}{Z^{\dagger}(S^-,S)}.
\end{eqnarray}
In this way we obtain
\begin{eqnarray}\label{eq:JSTderiv}
&& J(A;B)= \sum_{S,T\atop |S|=|T|}e^{-N(\sum \hat \alpha
+\sum\hat\beta)}
\frac{Z(S,T)Z(S^-,T^-)}{Z^{\dagger}(S,S^-)Z^{\dagger}(T,T^-)}
\\&&\qquad  \times \left. \nonumber \prod_{\alpha\in \overline{S}\atop \beta\in
\overline{T}}\frac{d}{d\alpha}\frac{d}{d\beta}
  \left(\frac{Z(\overline{S},\overline{T})Z(\overline{S},S^-)Z(\overline{T},T^-)Z(C,D)}
{Z(S,T)Z(C_{\overline{S}},S^-)Z(D_{\overline{T}},T^-)
Z(\overline{S},D)Z(\overline{T},C)
  }\right)\right|_{C=A\atop D=B}.\end{eqnarray}
Note that the sets $C_{\overline{S}}$ and $D_{\overline{T}}$ vary
from term to term in the sum over $S$ and $T$ since the division of
$C$ into the union of sets $C_S$ and $C_{\overline{S}}$ mimics the
form of $A=S+\overline{S}$, and similarly for $D$.  Also observe
that
\begin{eqnarray}
\frac{Z(\overline{S},\overline{T})Z(\overline{S},S^-)Z(\overline{T},T^-)Z(C,D)}
{Z(S,T)Z(C_{\overline{S}},S^-)Z(D_{\overline{T}},T^-)
Z(\overline{S},D)Z(\overline{T},C)
  }\bigg|_{C=A\atop D=B}=1.
  \end{eqnarray}
  To perform the   differentiations in (\ref{eq:JSTderiv}) we use a
form of logarithmic differentiation expressed in the following.

\begin{lemma}
\label{lemma:diff} Let $H$ be a differentiable function of $w\in
W$. Then
\begin{eqnarray}
\left(\prod_{w\in W}\frac{d}{dw}\right)e^H=  e^{H }
\sum_{W=W_1+\dots+W_r}H(W_1)\dots H(W_r)
\end{eqnarray}
where
\begin{eqnarray}
H(W)=\left(\prod_{w\in W} \frac{d}{dw}\right)H.
\end{eqnarray}
The sum is over all set partitions of     $W$ into disjoint sets
$W_j$.
\end{lemma}
In words this Lemma says that to perform a derivative with respect
to each variable once, we form all of the set partitions of the
complete set of variables and add up over these set partitions the
product of the partial derivatives of the exponent $H$ with
respect to each variable in each subset of the partition. This
lemma is obvious upon working a few examples.

  We apply this lemma with
\begin{eqnarray} \label{eqn:Hdef} &&
H=H^{A,B,C,D}_{S,T}:=\sum_{\alpha\in \overline{S}\atop \beta\in
\overline{T}} \log z(\alpha+\beta)+\sum_{\alpha\in \overline{S}\atop
\hat \alpha\in S} \log z(\alpha-\hat\alpha) + \sum_{\beta\in
\overline{T}\atop \hat \beta\in T}
  \log z(\beta-\hat\beta)\\
&& \qquad \qquad \qquad \qquad
  -
\sum_{\alpha\in \overline{S}\atop \delta\in D} \log
z(\alpha+\delta)-\sum_{\beta\in \overline{T}\atop \gamma\in C} \log
z(\beta+\gamma)\nonumber
  \end{eqnarray}
and so obtain, with
$H_{S,T}(W):=H^{A,B,C,D}_{S,T}(W)\big|_{C=A\atop
D=B}=\left(\left(\prod_{w\in W}
\frac{d}{dw}\right)H^{A,B,C,D}_{S,T}\right)\big|_{C=A\atop D=B}$,
\begin{eqnarray}
&& J(A;B)= \sum_{S,T\atop |S|=|T|}e^{-N(\sum \hat \alpha
+\sum\hat\beta)}
\frac{Z(S,T)Z(S^-,T^-)}{Z^{\dagger}(S,S^-)Z^{\dagger}(T,T^-)}\sum_{\overline{S}+
\overline{T}\atop = W_1+\dots +W_R}\prod_{r=1}^R H_{S,T}(W_r).
\end{eqnarray}
Strictly speaking, $H_{S,T}(W)$ depends on $A$ and $B$, but from now
on use of $H_{S,T}(W)$ will refer always to the expressions in
(\ref{eqn:Ha})-(\ref{eq:H3}) and these can be used without
specifically referring to $A$ and $B$.

By consideration of (\ref{eqn:Hdef}) it is clear that we can
restrict the subsets $W_r$ to be singletons or else pairs which
have precisely one $\alpha$ and one $\beta$. This follows from
some easy calculations.  Since
\begin{eqnarray}
H^{A,B,C,D}_{S,T}(\{\alpha\})=\sum_{\beta\in \overline{T}}
\frac{z'}{z}(\alpha+\beta)+\sum_{\hat \alpha \in S}
\frac{z'}{z}(\alpha-\hat\alpha) -\sum_{\delta \in D}
\frac{z'}{z}(\alpha+\delta),
\end{eqnarray}
we have
\begin{eqnarray} \label{eqn:Ha}
H_{S,T}(\{\alpha\})&=& \sum_{\hat \alpha\in S}
\frac{z'}{z}(\alpha-\hat\alpha) -\sum_{\hat \beta\in
T}\frac{z'}{z}(\alpha+\hat \beta), \;\;\; \alpha\notin S.
\end{eqnarray}
Similarly,
\begin{eqnarray} \label{eqn:Hb}
H_{S,T}(\{\beta\})&=& \sum_{\hat \beta\in T}
\frac{z'}{z}(\beta-\hat\beta) -\sum_{\hat \alpha\in
S}\frac{z'}{z}(\beta+\hat \alpha),\;\;\; \beta \notin T.
\end{eqnarray}
In addition
\begin{eqnarray} \label{eqn:Hab}
H_{S,T}(\{\alpha,\beta \})=
\left(\frac{z'}{z}\right)'(\alpha+\beta),\;\;\; \alpha,\beta \notin
S {\rm \;or\;} T.
\end{eqnarray}
Also, \begin{equation}\label{eq:H0} H_{S,T}(\emptyset)=1,
\end{equation}
 and
 \begin{equation}\label{eq:Haa}
H_{S,T}(\{\alpha,\alpha'\})=H_{S,T}(\{\beta,\beta'\})=0
\end{equation}
 and
 \begin{equation}\label{eq:H3}
  H_{S,T}(W)=0,\;\;\; {\rm \;if\;} |W|\ge 3.
  \end{equation}
\end{proof}

\subsection{Residue identity}\label{sect:res}

A key ingredient of the proof of $n$-correlation will be the
following residue identity for $J^*(A;B)$:

\begin{lemma} \label{lem:residue} Suppose that $\alpha^*\in A$ and $\beta^*\in B$.  Let $A'=A-\{\alpha^*\}$
and $B'=B-\{\beta^*\}.$ Then $J^*(A;B)$ has a simple pole at
$\alpha^*=-\beta^*$ with
\begin{eqnarray}
\operatornamewithlimits{Res}_{\alpha^*=-\beta^*} J^*(A;B)
=NJ^*(A';B')+J^*(A';B)+J^*(A' +\{-\beta^*\};B').
\end{eqnarray}
\end{lemma}

\begin{proof}
By Theorem \ref{theo:J} we have
\begin{eqnarray}
J^*(A,B)=\sum_{{S\subset A\atop T\subset B}\atop |S|=|T|}
D_{S,T}(\overline{S},\overline{T})
\end{eqnarray}
where throughout this proof (and this paper) $A=\overline{S}+S$,
$B=\overline{T}+T$, and
\begin{eqnarray}
  D_{S,T}(\overline{S},\overline{T})=Q(S,T)\sum_{\overline{S}+ \overline{T} =\sum W_r}\prod_{r} H_{S,T}(W_r).
\end{eqnarray}
Here the sum is over any collection of non-empty sets $W_1,W_2,\dots
$ which form a partition of  $\overline{S}+ \overline{T}$,
\begin{eqnarray}
   Q(S,T)=e^{-N(\sum_{\hat \alpha\in S} \hat \alpha +\sum_{\hat{\beta}\in T}\hat\beta)}
\frac{Z(S,T)Z(S^-,T^-)}{Z^{\dagger}(S,S^-)Z^{\dagger}(T,T^-)}
\end{eqnarray}
and $H_{S,T}(W)$ is defined in  (\ref{eqn:Ha})-(\ref{eq:H3}). We
claim that $D$, $Q$ and $H$ have the following properties:

\begin{description}

\item[P1]If $\alpha^*\in \overline{S}$ and $\beta^*\in \overline{T}$, then $Q(S,T)$ is
independent of $\alpha^*$ and $\beta^*$ and
\begin{eqnarray}
H_{S,T}(W)=\left\{ \begin{array}{ll}
\frac{1}{(\alpha^*+\beta^*)^2}+O(1)
& \mbox{ if $W=\{\alpha^*,\beta^*\}$ }\\
$O(1)$ & \mbox{ otherwise }
\end{array} \right.
\end{eqnarray}
\item[P2] If $\alpha^*\in S$ and $\beta^*\in \overline{T}$, then $Q(S,T)$ is
regular when $\alpha^*=-\beta^* $ and
\begin{eqnarray}
H_{S,T}(W)=\left\{ \begin{array}{ll}
\frac{1}{\alpha^*+\beta^*}+O(1)
& \mbox{ if $W=\{\beta^*\}$ }\\
$O(1)$ & \mbox{ otherwise }
\end{array} \right.
\end{eqnarray}
\item[P3] If $\alpha^*\in \overline{S}$ and $\beta^*\in T$, then $Q(S,T)$ is
regular when $\alpha^*=-\beta^*$ and
\begin{eqnarray}
H_{S,T}(W)=\left\{ \begin{array}{ll}
\frac{1}{\alpha^*+\beta^*}+O(1)
& \mbox{ if $W=\{\alpha^*\}$ }\\
$O(1)$ & \mbox{ otherwise }
\end{array} \right.
\end{eqnarray}
\item[P4] If $\alpha^*\in S$ and $\beta^*\in T$ and
$S'=S-\{\alpha^*\}$ and $T'=T-\{\beta^*\}$, then
$Q(S,T)=\big(\frac{-1}{(\alpha^*+\beta^*)^2}+O(1)\big)Q_1(S,T)$
where
\begin{eqnarray}
&&Q_1(S,T)=Q(S',T')\big(1-(\alpha^*+\beta^*)\big(N+
H_{S',T'}(\{\alpha^*\})|_{\alpha^*=-\beta^*}+
H_{S',T'}(\{\beta^*\})\big)\nonumber
\\
&&\qquad\qquad\qquad\qquad\qquad\qquad\qquad\qquad\qquad\qquad
\qquad\qquad+O(|\alpha^*+\beta^*|^2)\big)
\end{eqnarray}
and
\begin{eqnarray}
&&H_{S,T}(W)=H_{S',T'}(W)-(\alpha^*+\beta^*)\big(H_{S',T'}(W+\{\alpha^*\})_{\alpha^*=-\beta^*}
+H_{S',T'} (W+\{\beta^*\})\big)\nonumber
\\
&&\qquad\qquad\qquad\qquad\qquad\qquad\qquad\qquad\qquad\qquad
\qquad\qquad\quad+O(|\alpha^*+\beta^*|^2).
\end{eqnarray}
\end{description}
We show that the lemma follows from these four properties and then
prove that these properties hold in this situation. (We will later
demonstrate a proof along very similar lines when we treat the
$n$-correlation of the zeta-zeros.)

From these four properties we obtain four Laurent or Taylor
expansions of $D_{S,T}(\overline{S},\overline{T})$ as a function of
$\alpha^*$ in a neighborhood of $-\beta^*$:
\begin{itemize}
\item
   If $\alpha^*\in \overline{S}$ and $\beta^*\in \overline{T}$, then (with $\overline{S}'=\overline{S}-\{\alpha^*\}$
   and $\overline{T}'=\overline{T}-\{\beta^*\}$)
\begin{eqnarray}\label{eq:item1}
D_{S,T}(\overline{S},\overline{T})&=&
\left(\frac{1}{(\alpha^*+\beta^*)^2}+O(1)\right)
  Q(S,T) \sum_{\overline{S}'+ \overline{T}' =\sum W_r}\prod_{r}
  H_{S,T}(W_r)\nonumber\\
  &=& \left(\frac{1}{(\alpha^*+\beta^*)^2}+O(1)\right)D_{S,T}(\overline{S}',\overline{T}');
\end{eqnarray}
consequently,
$\operatornamewithlimits{Res}_{\alpha^*=-\beta^*}D_{S,T}(\overline{S},\overline{T})=0$.
\item If $\alpha^*\in S$ and $\beta^*\in \overline{T}$, then
\begin{eqnarray}
\operatornamewithlimits{Res}_{\alpha^*=-\beta^*}D_{S,T}(\overline{S},\overline{T})&=&Q(S'+\{-\beta^*\},T)\sum_{\overline{S}+
\overline{T}' =\sum W_r}\prod_{r} H_{S'+\{-\beta^*\},T}(W_r)
\\&=&D_{S'+\{-\beta^*\},T}(\overline{S},\overline{T}').\nonumber
\end{eqnarray}
\item If $\alpha^*\in \overline{S}$ and $\beta^*\in T$, then
\begin{eqnarray}
\operatornamewithlimits{Res}_{\alpha^*=-\beta^*}D_{S,T}(\overline{S},\overline{T})&=&Q(S,T)\sum_{\overline{S}'+
\overline{T} =\sum W_r}\prod_{r} H_{S,T}(W_r)
\\
&=&D_{S,T}(\overline{S}',\overline{T}).\nonumber
\end{eqnarray}
\item If $\alpha^*\in S$ and $\beta^*\in T$, then
\begin{eqnarray}\label{eq:item4} &&
Q_1(S,T)\sum_{\overline{S}+ \overline{T} =\sum W_r}\prod_r
H_{S,T}(W_r)=Q(S',T') \sum_{\overline{S}+ \overline{T} =\sum
W_r}\prod_r H_{S',T'}(W_r)\\
&&\quad\times\bigg(1+(\alpha^*+\beta^*)\Big(N+
H_{S',T'}(\{\alpha^*\})|_{\alpha^*=-\beta^*}+
H_{S',T'}(\{\beta^*\})\nonumber
\\
&&\qquad\qquad+\sum_r\frac{H_{S',T'}(W_r+\{\alpha^*\})|_{\alpha^*=-\beta^*}
+H_{S',T'}(W_r+\{\beta^*\})}{H_{S',T'}(W_r)}\Big)\nonumber\\
&&\qquad\qquad\qquad\qquad\qquad\qquad\qquad\qquad\qquad
+O(|\alpha^*+\beta^*|^2)\bigg).\nonumber
\end{eqnarray}
Therefore in this final case,
\begin{eqnarray}
\operatornamewithlimits{Res}_{\alpha^*=-\beta^*}D_{S,T}(\overline{S},\overline{T})=N
D_{S',T'}(\overline{S},\overline{T})+D_{S',T'}(\overline{S}+
\{-\beta^*\},\overline{T})+D_{S',T'}(\overline{S},\overline{T}+
\{\beta^*\}).
\end{eqnarray}
Note that (\ref{eq:item4}) can be written as
\begin{eqnarray}\label{eq:DDDD}
D_{S',T'}(\overline{S},\overline{T})\Big(1+O\big(|\alpha^*+\beta^*|\big)\Big).
\end{eqnarray}
\end{itemize}

By (\ref{eq:item1}) and (\ref{eq:item4}) the double poles in P1
and P4 cancel because
\begin{eqnarray}
\sum_{S\subset A, T\subset B\atop {|S|=|T|\atop \{\alpha^*\}\in S,
\{\beta^*\}\in T}}D_{S',T'}(\overline{S},\overline{T}) =
\sum_{S\subset A, T\subset B\atop {|S|=|T|\atop \{\alpha^*\}\notin
S, \{\beta^*\}\notin T}}D_{S,T}(\overline{S}',\overline{T}')
=J^*(A',B');
\end{eqnarray}
therefore, the pole at $\alpha^*=-\beta^*$ is simple.

Combining the four bullet-points above, we have (where as usual
the primed notation means that $\alpha^*$ or $\beta^*$ has been
removed from that set)
\begin{eqnarray}
&&\operatornamewithlimits{Res}_{\alpha^*=-\beta^*}
J^*(A;B)\nonumber
\\
&&\qquad=\sum_{S\subset A, T\subset B\atop {|S|=|T|\atop
\{\alpha^*\}\in S, \{\beta^*\}\in T}} (N D_{S',T'}
(\overline{S},\overline{T}))+D_{S',T'}(\overline{S}+
\{-\beta^*\},\overline{T})+D_{S',T'}(\overline{S},\overline{T}+
\{\beta^*\})
\\
&&\qquad \qquad\qquad+\sum_{S\subset A, T\subset B\atop
{|S|=|T|\atop \{\alpha^*\}\notin S, \{\beta^*\}\in T}}
D_{S,T}(\overline{S}',\overline{T})+ \sum_{S\subset A, T\subset
B\atop {|S|=|T|\atop \{\alpha^*\}\in S, \{\beta^*\}\notin T}}
D_{S'+\{-\beta^*\},T}(\overline{S},\overline{T}')\nonumber
\end{eqnarray}
Note that since $\alpha^*$ doesn't appear in any of the summands,
with the temporary convention that $A'=R+\overline{R}$, we can
relabel two of the sums as follows:
\begin{eqnarray}
&&\sum_{S\subset A, T\subset B\atop {|S|=|T|\atop \{\alpha^*\}\in S,
\{\beta^*\}\in T}}D_{S',T'}(\overline{S},\overline{T}+ \{\beta^*\})
+ \sum_{S\subset A, T\subset B\atop {|S|=|T|\atop \{\alpha^*\}\notin
S,
\{\beta^*\}\in T}} D_{S,T}(\overline{S}',\overline{T})\nonumber \\
&&\quad \qquad=\sum_{R\subset A', T\subset B\atop {|R|=|T|\atop
\{\beta^*\}\notin T}}D_{R,T}(\overline{R},\overline{T}) +
\sum_{R\subset A', T\subset B\atop {|R|=|T|\atop
\{\beta^*\}\in T}} D_{R,T}(\overline{R},\overline{T})\nonumber \\
&& \quad \qquad= J^*(A';B).
\end{eqnarray}
Similarly,
\begin{eqnarray}
&&\Bigg(\sum_{S\subset A, T\subset B\atop {|S|=|T|\atop
\{\alpha^*\}\in S, \{\beta^*\}\in T}}D_{S',T'}(\overline{S}+
\{\alpha^*\},\overline{T}) +\sum_{S\subset A, T\subset B\atop
{|S|=|T|\atop \{\alpha^*\}\in S, \{\beta^*\}\notin T}}
D_{S,T}(\overline{S},\overline{T}')\Bigg) \Bigg|_{\alpha^*=-\beta^*}
\nonumber
\\
&&\quad\qquad\quad=J^*(A;B')\big|_{\alpha^*=-\beta^*} =
J^*(A'+\{-\beta^*\};B').
\end{eqnarray}
Thus we have arrived at
\begin{eqnarray}
\operatornamewithlimits{Res}_{\alpha^*=-\beta^*} J^*(A;B)
=NJ^*(A';B')+J^*(A';B)+J^*(A' +\{-\beta^*\};B'),
\end{eqnarray}
which is the statement of the lemma.



Now we verify that properties P1 through P4 are satisfied in the
random matrix situation where we have
\begin{eqnarray}
   Q(S,T)=e^{-N(\sum_{\hat \alpha\in S} \hat \alpha +\sum_{\hat{\beta}\in T}\hat\beta)}
\frac{Z(S,T)Z(S^-,T^-)}{Z^{\dagger}(S,S^-)Z^{\dagger}(T,T^-)}
\end{eqnarray}
    and
\begin{equation}\label{eqn:Heval}
H_{S,T}(W)=\left\{\begin{array}{ll} \sum_{\hat{\alpha}\in
S}\frac{z'}{z}(\alpha-\hat \alpha)-\sum_{\hat \beta\in T}
\frac{z'}{z}(\alpha+\hat \beta)
  &\mbox{ if $W=\{\alpha\}\subset A-S$}\\
\sum_{\hat \beta \in T}\frac{z'}{z}(\beta-\hat \beta )-\sum_{\hat
\alpha \in S} \frac{z'}{z}(\beta+\hat \alpha)
  &\mbox{ if $W=\{\beta\}\subset B-T$}\\
\left(\frac{z'}{z}\right)'(\alpha +\beta) &  {\rm \;if\;} W=\{\alpha,\beta\} {\rm \;with\;}
{{\alpha\in A-S}\atop {\beta\in B-T}}\\
0&\mbox{ otherwise}
\end{array}
\right.
\end{equation}

 We will
start with the simplest case that $\alpha^*\in \overline{S}$,
$\beta^*\in \overline{T}$. The only polar term from
$\alpha^*=-\beta^*$ arises from a situation when one of the
partition parts $W_r=\{\alpha^*,\beta^*\}$ and there is a pole from
$H_{S,T}(W_r)=\left(\frac{z'}{z}\right)'(\alpha^*+\beta^*)$. Since
$\left(\frac{z'}{z}\right)'(x)=1/x^2+O(1)$ and $\alpha^*$ and
$\beta^*$ don't appear in $Q(S,T)$, this completes the proof of P1.

 Next, suppose that $\alpha^*\in
S$ and $\beta^*\in \overline{T}$. The only pole in
$D_{S,T}(\overline{S},\overline{T})$ occurs in the product of the
$H$ for $H_{S,T}(W_r)$ when $W_r=\{\beta^*\}$. We have
\begin{eqnarray}\label{eqn:innotin}
H_{S,T}(\{\beta^*\})=\sum_{\hat\beta \in T}\frac{z'}{z}(\hat
\beta-\beta^*)- \sum_{\hat \alpha\in S}\frac{z'}{z}(\beta^*+\hat
\alpha)
\end{eqnarray}
for which, when $\hat \alpha=\alpha^*$, the term
$-\frac{z'}{z}(\beta^*+\alpha^*)$ has a simple pole at $
\alpha^*=-\beta^*$ with residue 1.  $Q(S,T)$ is clearly regular at
$\alpha^*=-\beta^*$.

Next, when $\alpha^*\in \overline{S}$ and $\beta^*\in T$, the only
pole in the product of the $H$ occurs for $H_{S,T}(\{\alpha^*\})$.
We have
\begin{eqnarray}\label{eqn:innotin}
H_{S,T}(\{\alpha^*\})=\sum_{\hat\alpha \in S}\frac{z'}{z}(\hat
\alpha-\alpha^*)- \sum_{\hat \beta\in T}\frac{z'}{z}(\alpha^*+\hat
\beta)
\end{eqnarray}
for which, when $\hat \beta =\beta^*$, the term
$-\frac{z'}{z}(\alpha^*+\beta^*)$ has a simple pole at
$\alpha^*=-\beta^*$ with residue 1. $Q(S,T)$ does not depend on
$\alpha^*$.

  Finally, we consider the case
$\alpha^*\in S$ and $\beta^*\in T$. We have
\begin{eqnarray}
Q(S,T)=z(\alpha^*+\beta^*)z(-\alpha^*-\beta^*) Q_1(S,T)
\end{eqnarray}
where
\begin{eqnarray}
Q_1(S,T)&=&Q(S',T')\nonumber
\\
&&\qquad \times e^{-N(\alpha^*+\beta^*)}\frac{\prod_{\hat \beta\in
T'}z(\alpha^*+\hat \beta)z(-\alpha^*-\hat \beta)\prod_{\hat
\alpha\in S'} z(\hat \alpha+\beta^*)z(-\hat
\alpha-\beta^*)}{\prod_{\hat \alpha\in S'}z(\alpha^*-\hat
\alpha)z(\hat \alpha-\alpha^*)\prod_{\hat \beta\in T'}
z(\beta^*-\hat \beta)z(\hat \beta-\beta^*)}.
\end{eqnarray}
Note that
\begin{eqnarray}
Q_1(S,T)\big|_{\alpha^*=-\beta^*}=Q(S',T').
\end{eqnarray}
We are well on our way to verifying property P4 because
$z(\alpha^*+\beta^*)z(-\alpha^*-\beta^*)
=\frac{-1}{(\alpha^*+\beta^*)^2}+\frac{1}{12}+O(|\alpha^*+\beta^*|)$
and we have an expansion for $Q_1(S,T)$ in the neighborhood of
$\alpha^*=-\beta^*$:
\begin{eqnarray} \label{eqn:Q} \nonumber
Q_1(S,T)&=&
Q(S',T')\big(1-N(\alpha^*+\beta^*)+O(|\alpha^*+\beta^*|^2\big))\\
&&\qquad \times \bigg(1+(\alpha^*+\beta^*)\bigg( \sum_{\hat\alpha
\in S'}\Big( \frac{z'}{z}(\hat \alpha+\beta^*)-
\frac{z'}{z}(-\beta^*-\hat\alpha)\Big)\nonumber
\\
&&\qquad \qquad +\sum_{\hat \beta\in T'}
\Big(\frac{z'}{z}(-\beta^*+\hat \beta)-\frac{z'}{z}(\beta^*-\hat \beta)\Big)\bigg)
+O(|\alpha^*+\beta^*|^2)\bigg)\nonumber\\
&=&Q(S',T')\bigg(1-(\alpha^*+\beta^*)\big(N+H_{S',T'}(\{\alpha^*\})|_{\alpha^*=-\beta^*}\\
&&\qquad \qquad\qquad+
  H_{S',T'}(\{\beta^*\})\big)+O(|\alpha^*+\beta^*|^2)\bigg).\nonumber
\end{eqnarray}
Now we obtain an expansion for $H_{S,T}(W)$, where we remember that
$W\subset \overline{S}+\overline{T}$ and so $W$ does not contain
$\alpha^*$ or $\beta^*$. By (\ref{eqn:Heval}) we have that
\begin{eqnarray}
H_{S,T}(\{\alpha\})&=&\sum_{\hat{\alpha}\in
S}\frac{z'}{z}(\alpha-\hat \alpha)-\sum_{\hat \beta\in T}
\frac{z'}{z}(\alpha+\hat \beta)\nonumber\\
&=&  H_{S',T'}(\{\alpha\})+ \frac{z'}{z}(\alpha-\alpha^*)-
\frac{z'}{z}(\alpha+ \beta^*);
\end{eqnarray}
\begin{eqnarray}
H_{S,T}(\{\beta\})&=&\sum_{\hat \beta \in
T}\frac{z'}{z}(\beta-\hat \beta )-
\sum_{\hat \alpha \in S} \frac{z'}{z}(\beta+\hat \alpha)\nonumber\\
&=& H_{S',T'}(\{\beta\})+ \frac{z'}{z}(\beta-\beta^*)-
\frac{z'}{z}(\beta+ \alpha^*);
\end{eqnarray}
and
\begin{eqnarray}
H_{S,T}(\{\alpha,\beta\})=\left(\frac{z'}{z}\right)'(\alpha
+\beta)=H_{S',T'}(\{\alpha,\beta\}).
\end{eqnarray}
Thus,
\begin{eqnarray}
H_{S,T}(W)\bigg|_{\alpha^*=-\beta^*}=H_{S',T'}(W)
\end{eqnarray}
and
\begin{equation}
\frac{d}{d\alpha^*} H_{S,T}(W)\bigg|_{\alpha*=-\beta^*}=\left\{
\begin{array}{ll}
  -\left(\frac{z'}{z}\right)'(\alpha+\beta^*)
  & \mbox{ if $W=\{\alpha\}$ }\\
  -\left(\frac{z'}{z}\right)'(\beta-\beta^*)& \mbox{ if $W=\{\beta\}$} \\
0&\mbox{ otherwise}
\end{array}\right\}.
\end{equation}
Note that we can write this as
\begin{equation}
\frac{d}{d\alpha^*}
H_{S,T}(W)\bigg|_{\alpha*=-\beta^*}=-H_{S',T'}(W+\{\alpha^*\})\big|_{\alpha^*=-\beta^*}
-H_{S',T'}(W+\{\beta^*\}),
\end{equation}
where one or both of the terms will be zero.  From this the
expansion of $H_{S,T}(W)$ in P4 follows.

This concludes the proof of Lemma \ref{lem:residue}.
  \end{proof}

\subsection{$n$-correlation via the ratios theorem}

In this section we will prove the following expression for the
$n$-correlation.

\begin{theorem} \label{theo:offtheline}
Let $\mathcal C_-$ denote the path  from $-\delta+\pi i$ down to
$-\delta-\pi i$ and  let $\mathcal C_+$ denote the path  from
$\delta-\pi i$ up to $\delta+\pi i$ and let $f$ be a
$2\pi$-periodic, holomorphic function of $n$ variables. Using the
notation $J(A;B)$ from Theorem \ref{theo:J},
\begin{eqnarray}\label{eq:offtheline}&&\int_{U(N)}\sum_{j_1,\dots ,j_n=1}^N
f(\theta_{j_1},\dots,\theta_{j_n})dX\nonumber
\\
&&\qquad =\frac{1}{(2\pi i)^n} \sum_{K+L+M=
\{1,\dots,n\}}(-1)^{|L|+|M|} N^{|M|} \\
&&\qquad\qquad\qquad \times\int_{\mathcal {C_+}^K} \int_{\mathcal
{C_-}^{L+ M}}J(z_K;-z_L) f(iz_1,\dots,iz_n)~dz_1\dots
~dz_n\nonumber
\end{eqnarray}
where
  $z_K=\{z_k:k\in K\}$,  $-z_L=\{-z_\ell:\ell\in L\}$ and $\int_{\mathcal {C_+}^K} \int_{\mathcal {C_-}^{L+ M}}$
  means that we are integrating all of the variables in $z_K$ along the $\mathcal C_+$
  path  and all of the variables in $z_{L}$ or $z_{M}$ along the $\mathcal C_-$
  path.
  \end{theorem}

\begin{proof}
 Since
\begin{eqnarray}
g(z)=\Lambda_X(e^z)=\prod_{j=1}^N\left(1-e^ze^{-i\theta_j}\right)
\end{eqnarray}
has zeros at $z=i\theta_j+2\pi i m$, $m\in\mathbb Z$, by Cauchy's
theorem we can express a sum
\begin{eqnarray}
\sum_{j=1}^Nf(\theta_j)=\frac{1}{2\pi i}\int_{\mathcal
C}\frac{g'}{g}(z)f(z/i)~dz =\frac{1}{2\pi i}\int_{\mathcal C}e^z
\frac{\Lambda_X'}{\Lambda_X}(e^z)f(z/i)~dz
\end{eqnarray}
where $\mathcal C$ is a positively oriented contour which encloses
a subinterval of the imaginary axis of length $2\pi$. We choose a
specific path $\mathcal C$ to be the positively oriented rectangle
that has vertices $\delta-\pi i,\delta+\pi i, -\delta+\pi i,
-\delta-\pi i$ where $\delta $ is a small positive number. More
generally, we have
\begin{eqnarray}
&&\sum_{j_1,\dots
,j_n=1}^Nf(\theta_{j_1},\dots,\theta_{j_n})\nonumber
\\
&&\qquad\qquad=\frac{1}{(2\pi i)^n} \int_{\mathcal C}\dots
\int_{\mathcal C} \prod_{j=1}^n e^{z_j}
\frac{\Lambda_X'}{\Lambda_X}(e^{z_j})
f(z_1/i,\dots,z_n/i)~dz_1\dots dz_n.
\end{eqnarray}
We average this equation over $X\in U(N)$ and,  after a change of
variables $z_j\to -z_j$, we obtain
\begin{eqnarray} \label{eqn:basic}
&&\int_{U(N)}\sum_{j_1,\dots
,j_n=1}^Nf(\theta_{j_1},\dots,\theta_{j_n})dX\nonumber\\
&&\qquad\qquad=\frac{1}{(2\pi i)^n} \int_{\mathcal
C^n}J(z_1,\dots,z_n;) f(iz_1,\dots,iz_n)~dz_1\dots dz_n.
\end{eqnarray}

Let $\mathcal C_-$ denote the path along the left side of
$\mathcal C$ from $-\delta+\pi i$ down to $-\delta-\pi i$ and  let
$\mathcal C_+$ denote the path along the right side of $\mathcal
C$ from $\delta-\pi i$ up to $\delta+\pi i$.  Since the
periodicity of the function $f$ implies that the horizontal
segments of the contours cancel, each variable $z_j$ is on one or
the other of these two vertical paths. Thus, our expression is a
sum of $2^n $ terms, each term being an n-fold integral with each
integral on a vertical line segment either $\mathcal C_-$ or
$\mathcal C_+.$ For each variable $z_j$ which is on $\mathcal C_-$
we use the functional equation (\ref{eqn:fe}) to replace $
e^{-z_j}\frac{\Lambda_X'}{\Lambda_X}(e^{-z_j}) $  by
$N-e^{z_j}\frac{\Lambda_{X^*}'}{\Lambda_{X^*}}(e^{z_j}). $ In this
way we find that
\begin{eqnarray}&&
\frac{1}{(2\pi i)^n}\int_{\mathcal C^n}J(z_1,\dots,z_n;)
f(iz_1,\dots,iz_n)~dz_1\dots dz_n\nonumber\\
&&  \qquad = \frac{1}{(2\pi
i)^n}\sum_{\epsilon_j\in\{-1,+1\}}\int_{\mathcal
C_{\epsilon_1}}\dots \int_{\mathcal C_{\epsilon_n}}\int_{U(N)}
\;(-1)^n\; \prod_{j=1}^n \left(\frac{1-\epsilon_j}{2}N+\epsilon_j
e^{-\epsilon_j z_j}\frac{\Lambda_{X^{\epsilon_j}}'}
{\Lambda_{X^{\epsilon_j}}}
(e^{-\epsilon_j z_j})\right)\\
&&\qquad \qquad \times f(iz_1,\dots,iz_n)~dXdz_1\dots
dz_n.\nonumber
\end{eqnarray}

Another way to write this equation is
\begin{eqnarray}&&
\frac{1}{(2\pi i)^n}\int_{\mathcal C^n}J(z_1,\dots,z_n;)
f(iz_1,\dots,iz_n)~dz_1\dots dz_n\nonumber\\
&&  \qquad= \frac{1}{(2\pi i)^n}\int_{U(N)}\;(-1)^n
\sum_{K\subset\{1,\dots,n\}} \prod_{j\in K} \int_{\mathcal
C_+}e^{-z_j}\frac{\Lambda_X'} {\Lambda_X} (e^{- z_j})
\prod_{j\notin K}\int_{\mathcal C_-}
  \left(N-e^{z_j}\frac{\Lambda_{X^*}'}
{\Lambda_{X^*}}
(e^{ z_j})\right)\\
&& \qquad \qquad \times f(iz_1,\dots,iz_n)dz_1\dots
dz_n~dX.\nonumber
\end{eqnarray}

The expansion of the product over $j\notin K$ can be easily
expressed as a sum over further subsets of $\{1,\ldots,n\}$.  We
have
\begin{eqnarray}&&
\frac{1}{(2\pi i)^n}\int_{\mathcal C^n}J(z_1,\dots,z_n;)
f(iz_1,\dots,iz_n)~dz_1\dots dz_n\nonumber\\
&&  \qquad= \frac{1}{(2\pi i)^n}\int_{U(N)}\;(-1)^n
\sum_{K+L+M=\{1,\dots,n\}} \prod_{j\in K} \int_{\mathcal
C_+}e^{-z_j}\frac{\Lambda_X'} {\Lambda_X} (e^{- z_j})\prod_{j\in
L}\int_{\mathcal C_-}
  (-1)e^{z_j}\frac{\Lambda_{X^*}'}
{\Lambda_{X^*}} (e^{ z_j})
  \\
&& \qquad \qquad \times \prod_{j\in M}\int_{\mathcal C_-}N
  f(iz_1,\dots,iz_n)~dz_1\dots ~dz_n~dX.\nonumber
\end{eqnarray}
Using this last equation, (\ref{eqn:basic}) and the definition of
$J(A;B)$ from Theorem \ref{theo:J}, we have the statement of the
Theorem.

\end{proof}



\subsection{$n$-correlation theorem}

We will now prove our main theorem.

\begin{theorem}   Let $J^*$ be as defined in Theorem \ref{theo:J}. Then
\label{theo:main}
\begin{eqnarray}&&\int_{U(N)}\sideset{}{^*}\sum_{1\le j_1,\dots , j_n\le N} f(\theta_{j_1},\dots,\theta_{j_n})
dX_N\nonumber
\\
&&\qquad =\frac{1}{(2\pi )^n}\int_{[0,2\pi]^n} \sum_{K+L+M=
\{1,\dots,n\}} N^{|M|}
  J^*(-i\theta_K;i\theta_L)
f(\theta_1,\dots,\theta_n)~d\theta_1\dots ~d\theta_n
\end{eqnarray}
where $i\theta_L=\{i\theta_\ell:\ell\in L\}$,
$-i\theta_K=\{-i\theta_k:k\in K\}$ and the star on the sum
indicates summation over distinct indices.  Moreover, the
integrand has no poles on the path of integration.
\end{theorem}

Note the similar forms of Theorem \ref{theo:offtheline} and
Theorem \ref{theo:main}.  In the former the sum is over all
indices and the integrals are on paths slightly shifted away from
the imaginary axis and in the latter the sum is over distinct
indices and the integration is along the imaginary axis.  Moving
the integrals onto the imaginary axis results in some principal
value terms, and surprisingly these cancel exactly with extra
terms in the sum in \ref{theo:offtheline}.

We actually prove a more general theorem (Theorem \ref{theo:main1}
below). We start with a little notation: For a given $n$ and $0\leq
R\leq n$, let the sum $\sum_{j_1,\ldots,j_n=1}^N$ with the
additional condition that $j_m\neq j_{\ell}$ if $m>R$ {\em and}
$\ell>R$ be denoted by $\sum^{n,R}$.  If we additionally fix three
disjoint sets $K$, $L$ and $M$ whose union is $\{1,2,\ldots,n\}$,
then we introduce the following notation for the familiar integral
\begin{eqnarray}\label{eq:Jfdef}
&&\int_{-\pi i}^{\pi i} \cdots \int_{-\pi i}^{\pi i}
\int_{C_+^{K\cap \{1,\ldots,R\}}} \int_{C_-^{(L + M)\cap
\{1,\ldots,R\}}} J^*(z_K;-z_L)f(iz_1,\ldots,iz_n) dz_1\cdots
dz_R\;dz_{R+1} \ldots dz_{n}\nonumber \\
&&\qquad =:I_{f;K,L,M}^{n,R}.
\end{eqnarray}
Once again, the integrals on the imaginary axis are principal
value integrals.

We have already derived equation (\ref{eq:offtheline}).  In the
new notation this is written as
\begin{eqnarray}
\label{eq:offtheline2} &&(2\pi i)^n \int_{U(N)} \sum\!^{n,n}
f(\theta_{j_1},\ldots,\theta_{j_n}) dX = \sum_{K+L+M=\{1,\ldots,n\}}
(-1)^{|L + M|}N^{|M|}I^{n,n}_{f;K,L,M}.
\end{eqnarray}
Note that (\ref{eq:offtheline}) features $J$ whereas
$I^{n,R}_{f;K,L,M}$ is defined in terms of $J^*$.  However, when
$R=n$ (that is, all the integrals are off the imaginary axis)
Theorem \ref{theo:J} says that $J$ and $J^*$ are equal.


With the help of Lemma \ref{lem:residue} we will prove the
following:
\begin{theorem}
\label{theo:main1} Using the notation of (\ref{eq:Jfdef}) and the
preceding paragraph, with $0\leq R \leq n$,
\begin{eqnarray}
\label{eq:main1} &&(2\pi i)^n \int_{U(N)} \sum\!^{n,R}
f(\theta_{j_1},\ldots,\theta_{j_n}) dX = \sum_{K+L+M=\{1,\ldots,n\}}
(-1)^{|(L + M)\cap\{1,\ldots,R\}|}N^{|M|}I^{n,R}_{f;K,L,M}.
\end{eqnarray}
\end{theorem}

\begin{proof}  We will prove this by induction.  Assume that Theorem
\ref{theo:main1} holds for $n-1$ and any $0\leq R\leq n-1$.

We start with the right side of (\ref{eq:main1}) and move the
$z_R$ integral onto the imaginary axis, resulting in a principal
value integral and a residue at $z_R=z_t$, for $t>R$, in any term
where $R\in K$, $t\in L$ {\em or} $R\in L, t\in K$. A close
inspection of the integral and the form of $J^*(z_K;-z_L)$ reveals
that there is no pole unless $R$ and $t$ are in one of these two
configurations (see the comment in the final paragraph of Section
\ref{sect:ratiolambda}). Also, if $t<R$ then the contour on which
$z_t$ is integrated has not yet been moved and so it remains on
the far side of the imaginary axis from the $z_R$ contour and
hence does not yield a pole. Each residue contribution comes in
the form of the three terms in Lemma \ref{lem:residue}, multiplied
by $\pi i$. (It is $\pi i$ rather than $2\pi i$ because the $z_R$
contour is moving precisely onto the imaginary axis, where $z_t$
lies, yielding half the contribution of a contour completely
encircling the pole.) Thus
\begin{eqnarray}
\label{eq:indproof1} &&\sum_{K+L+M=\{1,\ldots,n\}} (-1)^{|(L+
M)\cap\{1,\ldots,R\}|}N^{|M|}I^{n,R}_{f;K,L,M}\nonumber \\
&&= \sum_{K+L+M=\{1,\ldots,n\}} (-1)^{|(L+
M)\cap\{1,\ldots,R-1\}|}N^{|M|}I^{n,R-1}_{f;K,L,M}\\
&&\qquad+2\times\sum_{t=R+1}^n \pi i \Big[
\sum_{K'+L'+M=\{1,\ldots,n\}-\{R,t\}} (-1)^{|(L' +
M)\cap\{1,\ldots,R-1\}|}N^{|M|} \nonumber \\
&&\qquad\qquad\times\int_{-\pi i}^{\pi i} \cdots \int_{-\pi
i}^{\pi i} \int_{C_+^{K\cap \{1,\ldots,R-1\}}} \int_{C_-^{(L+
M)\cap \{1,\ldots,R-1\}}}
\Big(J^*(z_{K'+\{t\}};-z_{L'})\nonumber\\
&&\qquad\qquad\qquad\qquad+J^*(z_{K'};-z_{L'+\{t\}})+NJ^*(z_{K'};-z_{L'})\Big)\nonumber\\
&&\qquad\qquad\times
f(iz_1,\ldots,iz_{R-1},iz_t,iz_{R+1},\ldots,iz_n) dz_1\cdots
dz_{R-1}\;dz_{R+1} \ldots dz_{n}\Big]\nonumber
\end{eqnarray}
The final sum above contains the two identical contributions from
the case $R\in K, t\in L$ and the case $R\in L, t\in K$.  To
confirm the sign of each term, if $R\in K$, the residue is
multiplied by $+i\pi$ because the contour of integration moves in
from the right of the imaginary axis (skirting the pole in the
positive direction) and the argument $z_R$ in $J^*(z_K;-z_L)$
occurs with a plus sign. Note that $(-1)^{|(L'+
M)\cap\{1,\ldots,R-1\}|}=(-1)^{|(L+ M)\cap\{1,\ldots,R\}|}$ if
$R\in K$ and $L=L'+\{t\}$.  On the other hand, if $R\in L$ then
the $z_R$ contour comes from the left of the imaginary axis, but
as $C_-$ is directed downwards, the pole is still circled in the
positive direction.  However, $z_R$ appears in $J^*(z_K;-z_L)$
with a minus sign, so the residue acquires an extra minus sign,
which is captured above because $(-1)^{|(L'+
M)\cap\{1,\ldots,R-1\}|}=(-1)\times(-1)^{|(L+
M)\cap\{1,\ldots,R\}|}$ if $L=L'+\{R\}$.

In the integrals in the final sum above we now relabel the
integration variables \linebreak
$z_1,z_2,\ldots,z_{R-1},z_{R+1},\ldots,z_n$ by
$z_1,z_2,\ldots,z_{n-1}$ so that
$f(z_1,\ldots,z_{R-1},z_t,z_{R+1},\ldots,z_t,\ldots,z_n)$ is
replaced by
$f(z_1,\ldots,z_{R-1},z_{t-1},z_R,\ldots,z_{t-1},\ldots,z_{n-1})$
$=:g_t(z_1,\ldots,z_{n-1})$. In addition, for some function $h$ of
sets $K,L$ and $M$,
\begin{eqnarray}
&&\sum_{K+L+M=\{1,\ldots,m-1\}}
(h(K+\{m\},L,M)+h(K,L+\{m\},M)+h(K,L,M+\{m\}))\nonumber
\\
&&\qquad\qquad\qquad=\sum_{K+L+M=\{1,\ldots,m\}} h(K,L,M),
\end{eqnarray}
so we now rewrite the three $J^*$ terms in the final sum in
(\ref{eq:indproof1}) as a sum over partitions of
$\{1,\ldots,n-1\}$. Thus (\ref{eq:indproof1}) equals
\begin{eqnarray}
\label{eq:indproof2} && \sum_{K+L+M=\{1,\ldots,n\}} (-1)^{|(L+
M)\cap\{1,\ldots,R-1\}|}N^{|M|}I^{n,R-1}_{f;K,L,M} \nonumber\\
&&\quad+2\pi i\sum_{t=R+1}^n  \sum_{K+L+M=\{1,\ldots,n-1\}}
(-1)^{|(L+ M)\cap\{1,\ldots,R-1\}|}N^{|M|}I_{g;K,L,M}^{n-1,R-1}.
\end{eqnarray}

By the induction hypothesis, this equals
\begin{eqnarray}
&& \sum_{K+L+M=\{1,\ldots,n\}} (-1)^{|(L+
M)\cap\{1,\ldots,R-1\}|}N^{|M|}I^{n,R-1}_{f;K,L,M} \nonumber\\
&&\quad+2\pi i\sum_{t=R+1}^n  (2\pi i)^{n-1}\int_{U(N)}
\sum\!^{n-1,R-1} g_t(\theta_{j_1},\ldots,\theta_{j_{n-1}})dX.
\end{eqnarray}

Note that the left side of (\ref{eq:main1}) can be written as
\begin{eqnarray}
&&(2\pi i)^n \int_{U(N)}\sum\!^{n,R-1}
f(\theta_{j_1},\ldots,\theta_{j_n})dX \nonumber\\
&&\qquad\qquad+ (2\pi i)^n\int_{U(N)}\sum_{t=R+1}^n
\sum\!^{n-1,R-1} g_t(\theta_{j_1},\ldots,\theta_{j_{n-1}})dX,
\end{eqnarray}
where the second sum incorporates all the terms where
$\theta_{j_R}=\theta_{j_t}$, $t>R$, and then uses the same
relabelling of the variables $\theta_{j_1},\theta_{j_2},\ldots,
\theta_{j_{R-1}},\theta_{j_{R+1}},\ldots,\theta_{j_n}$ and the
definition of $g_t$ as described before (\ref{eq:indproof2}).
Therefore
\begin{eqnarray}
&&(2\pi i)^n \int_{U(N)} \sum\!^{n,R-1}
f(\theta_{j_1},\ldots,\theta_{j_n}) dX \nonumber \\
&&\qquad\qquad= \sum_{K+L+M=\{1,\ldots,n\}} (-1)^{|(L+
M)\cap\{1,\ldots,R-1\}|}N^{|M|}I^{n,R-1}_{f;K,L,M}
\end{eqnarray}
and so, using the induction hypothesis, we have used
(\ref{eq:main1}) for a given $n$ and $R$ to deduce the same
expression for $n$ and $R-1$. Since in (\ref{eq:offtheline2}) we
have derived the expression for $R=n$ for any $n$, we have shown
that if (\ref{eq:main1}) is true for $n-1$, it is also true for $n$.
To justify the induction hypothesis in $n$, we consider $n=1$.
Equation (\ref{eq:offtheline}) states
\begin{equation}
2\pi i\int_{U(N)} \sum\!^{1,1}f(\theta_{j_1}) dX=
\sum_{K+L+M=\{1\}}(-1)^{|(L+ M)\cap \{1\}|} N^{|M|}
I^{1,1}_{f;K,L,M}=-NI_{f;\emptyset,\emptyset,\{1\}}^{1,1}.
\end{equation}
The final step above follows by remembering that
$J^*(\emptyset;z_A)=0$ for any nonempty set $A$. Since
$\sum^{1,1}=\sum^{1,0}$ and
$I_{f;\emptyset,\emptyset,\{1\}}^{1,1}=-I_{f;\emptyset,\emptyset,\{1\}}^{1,0}$,
it is immediate that
\begin{equation}
2\pi i\int_{U(N)} \sum\!^{1,0}f(\theta_{j_1}) dX=
\sum_{K+L+M=\{1\}}(-1)^{|(L+ M)\cap \emptyset|} N^{|M|}
I^{1,0}_{f;K,L,M}.
\end{equation}
This completes the proof of Theorem \ref{theo:main1}.
\end{proof}

It remains to verify that the integrand in Theorem \ref{theo:main}
has no poles on the path of integration.  We have already
confirmed in Lemma \ref{lem:residue} that each
$J^*(-i\theta_K;i\theta_L)$ has only a simple pole at
$\theta_k=-\theta_\ell$ for $\theta_k\in \theta_K$ and
$\theta_\ell\in \theta_L$.

We check that
\begin{eqnarray}\label{eq:singset}
\sum_{K+L+M=\{1,2,\ldots,n\}} N^M J^*(-i\theta_K;i\theta_L)
\end{eqnarray}
has no pole at $\theta_1=\theta_2$ for generic values of the
remaining variables.  A given $J^*(-i\theta_K;i\theta_L)$ only has
a pole when $\theta_1\in \theta_L$ and $\theta_2 \in \theta_K$, or
vice versa, so
\begin{eqnarray}
&&\operatornamewithlimits{Res}_{\theta_1=\theta_2}
\sum_{K+L+M=\{1,2,\ldots,n\}} N^M
J^*(-i\theta_K;i\theta_L) \nonumber \\
&&\qquad= \sum_{K+L+M=\{3,\ldots,n\}} N^M
\operatornamewithlimits{Res}_{\theta_1=\theta_2}\Big(J^*(-i\theta_K+\{-i\theta_1\}
;\{i\theta_2\}+i\theta_L) \nonumber \\
&&\qquad\qquad\qquad\qquad+ J^*(-i\theta_K+\{-i\theta_2\}
;\{i\theta_1\}+i\theta_L)\Big)=0;
\end{eqnarray}
this is zero because $\operatornamewithlimits{Res}_{s=x}f(s,x)=-
\operatornamewithlimits{Res}_{s=x}f(x,s)$.  Thus if
(\ref{eq:singset}) had a singular set it would be of complex
dimension less than $n-1$ and this implies that there is no
singular set (see for example \cite{kn:krantz}, Corollary 7.3.2).

Our new proof of $n$-correlation in the case of random matrix
theory is now complete.

\section{Correlations of the Riemann zeros}

Now we turn to the Riemann zeta-function. The goal is to obtain a
precise conjecture for the $n$-correlation of its zeros and we do
this following the method of the previous section for the random
matrix case.

\subsection{The ratios conjecture for the zeta-function}

 We derive our formula rigorously from the ratios
conjecture for the zeta-function, which we now state.

\begin{conjecture}[Ratios Conjecture \cite{kn:cfz2}]\label{conj:ratiozeta}
Let $Z_\zeta(A,B)=\prod_{\alpha\in A\atop\beta\in B}
\zeta(1+\alpha+\beta)$ and
\begin{eqnarray}Z_\zeta(A,B;C,D):=
\frac{Z_\zeta(A,B)Z_\zeta(C,D)}
{Z_\zeta(A,D)Z_\zeta(B,C)}.\end{eqnarray}
  Further, let
  \begin{eqnarray}\label{eq:Azeta}
  \mathcal{A}_\zeta(A,B;C,D)=\prod_p Z_{p}(A,B;C,D)\int_0^1\mathcal{A}_{p,\theta}(A,B;C,D)~d\theta
  \end{eqnarray}
  where
$z_p(x):=(1-p^{-x})^{-1}$, $Z_p(A,B)=\prod_{\alpha\in
A\atop\beta\in B} z_p(1+\alpha+\beta)^{-1}$ and
\begin{eqnarray}Z_p(A,B;C,D):=
\frac{Z_p(A,B)Z_p(C,D)} {Z_p(A,D)Z_p(B,C)}\end{eqnarray}
  and
\begin{eqnarray}
\label{eq:Aptheta} \mathcal{A}_{p,\theta}(A,B;C,D):= \frac
{\prod_{\alpha\in A} z_{p,-\theta}(\frac 12 +\alpha)
\prod_{\beta\in B}z_{p,\theta}(\frac 12 +\beta)}{ \prod_{\gamma\in
C}z_{p,-\theta}(\frac 12 +\gamma) \prod_{\delta\in
D}z_{p,\theta}(\frac 12 +\delta)}
\end{eqnarray}
with $z_{p,\theta}(x):=(1-e(\theta)p^{-x})^{-1}$. Then, provided
that $-\frac{1}{4}<\Re \alpha,\Re \beta<\frac{1}{4}$,
$\frac{1}{\log T} \ll \Re \gamma,\Re \delta<\frac{1}{4}$ and $\Im
\alpha,\Im\beta,\Im\gamma,\Im\delta\ll T$, we conjecture that,
with $s=\frac12+it$, for any interval $I\subset [-T,T]$,
\begin{eqnarray}
&&\int_I \frac{\prod_{\alpha\in A}\zeta(s+\alpha)\prod_{\beta\in
B} \zeta(1-s+\beta)}{\prod_{\gamma\in C}\zeta(s+\gamma)
\prod_{\delta\in D}\zeta(1-s+\delta)} ~dt  = \int_I \mathcal
R_{\zeta,t}(A,B;C,D) ~dt +O(|I|^{1/2+\epsilon})
\end{eqnarray}
where
\begin{eqnarray}
\mathcal R_{\zeta,t}(A,B;C,D)=\sum_{S\subset A,T\subset B\atop
|S|=|T|} X_t(S,T) Z_\zeta \mathcal{A}_\zeta(\overline{S}+
T^-,\overline{T}+ S^-;C,D).
\end{eqnarray}
Here $T^-$  means the set of all of the negatives of elements of
$T$ (i.e. $T^-:=\{-t:t\in T\}$), $A=S+\overline{S}$,
$B=T+\overline{T}$ and
\begin{eqnarray}
X_t(S,T)=\prod_{\hat{\alpha}\in
S}\chi(s+\hat{\alpha})\prod_{\hat{\beta}\in
T}\chi(1-s+\hat{\beta}),
\end{eqnarray}
where $\chi(1-s)=\chi(s)^{-1}=2(2\pi)^{-s}\Gamma(s)\cos \frac {\pi
s}2$ is the factor from the functional equation
$\zeta(s)=\chi(s)\zeta(1-s)$.
\end{conjecture}
\begin{remark}Note
that since $|S|=|T|$, for small shifts $\hat{\alpha}$ and
$\hat{\beta}$ we have
\begin{eqnarray}\label{eq:chiapprox}
X_t(S,T)= e^{-\ell(\sum_{\hat\alpha\in S} \hat \alpha +\sum_{\hat
\beta\in T}\hat\beta)}\bigg(1+O(1/(1+|t|)\bigg),
\end{eqnarray}
where $\ell =\log\frac{t}{2\pi}$, which can sometimes be used to
simplify formulae.
\end{remark}

The method for constructing the ratios conjecture is detailed in
\cite{kn:cfz2} and is based on the same principles as the recipe
for generating conjectures for moments (see \cite{kn:cfkrs}) of
zeta and $L$-functions.  (Moments cover just the case where
$C=\{\emptyset\}$ and $D=\{\emptyset\}$.)

\begin{corollary}\label{cor:chis}
With the same conditions on $\alpha,\beta,\gamma$ and $\delta$ as
in Conjecture \ref{conj:ratiozeta}, and with conditions on $\mu$
the same as those on $\alpha$ and $\beta$, we have
\begin{eqnarray}
&&\int_I \frac{\prod_{\alpha\in A}\zeta(s+\alpha)\prod_{\beta\in
B} \zeta(1-s+\beta)}{\prod_{\gamma\in C}\zeta(s+\gamma)
\prod_{\delta\in D}\zeta(1-s+\delta)} \prod_{\mu \in  U}
\frac{\chi'}{\chi}(s+\mu)~dt \nonumber \\
&&\qquad= \int_I \mathcal R_{\zeta,t}(A,B;C,D)\prod_{\mu \in  U}
\frac{\chi'}{\chi}(s+\mu) ~dt +O(|I|^{1/2+\epsilon})
\end{eqnarray}
as a consequence of Conjecture \ref{conj:ratiozeta}.
\end{corollary}

\begin{proof}
This follows immediately by integration by parts using the fact
that
\begin{eqnarray}
&&\frac{\chi'}{\chi} (s) \ll \log(2+|s|)\quad \mbox{ and } \quad
\frac{d}{ds} \frac{\chi'}{\chi}(s)\ll \frac{1}{1+|s|}.
\end{eqnarray}
\end{proof}

\subsection{Averages of logarithmic derivatives of the Riemann
zeta function}

To determine the correlations of the Riemann zeros, we will need a
result about averaging logarithmic derivatives of the zeta
function:

\begin{theorem} \label{theo:Jzeta} Assuming the Ratios Conjecture, if $\Re \alpha_i,\Re
\beta_j>0$ for $\alpha_i\in A$ and $\beta_j \in B$ then
$J_{\zeta,I}(A;B;U)=J^*_{\zeta,I}(A;B;U)+O(|I|^{1/2+\epsilon})$ where for an interval $I$
\begin{eqnarray} \label{eq:defJzeta}
&&J_{\zeta,I}(A;B;U)\\&&\qquad :=\int_I \prod_{\alpha \in A}
\frac{\zeta'}{\zeta}(\tfrac 12 +it+\alpha)\:\prod_{\beta\in
B}\frac{\zeta'}{\zeta}(\tfrac 12 -it+\beta)\prod_{\mu\in
U}\Big(-\frac{\chi'}{\chi}(\tfrac 12 +it + \mu)\Big)~dt\nonumber
\end{eqnarray}
and
\begin{eqnarray}
&& J_{\zeta,I}^*(A;B;U) :=\int_I J_{\zeta,t}^*(A;B;U)~dt,
\end{eqnarray}
where
\begin{eqnarray}
 &&J_{\zeta,t}^*(A;B;U):=\sum_{S\subset A,T\subset B\atop
|S|=|T|}X_t(S,T)   \nonumber
\frac{Z_\zeta(S,T)Z_\zeta(S^-,T^-)}{{Z_\zeta}^{\dagger}(S,S^-){Z_\zeta}^{\dagger}(T,T^-)}\mathcal{A}_\zeta(T^-,S^-;S,T)
\\&& \qquad \qquad \qquad\qquad  \times
\sum_{{\overline{S}+ \overline{T}\atop = W_1+\dots + W_R}
}\prod_{r=1}^R \mathcal{H}_{S,T}(W_r) \times \prod_{\mu\in
U}\Big(-\frac{\chi'}{\chi}(1/2+it+\mu)\Big).
\end{eqnarray}
Here we use the notation $Z_{\zeta}$ as in Conjecture
\ref{conj:ratiozeta} and $Z^{\dagger}_\zeta(A,B)=\prod_{{\alpha\in
A\atop\beta\in B} \atop{\alpha+\beta\neq 0}}
\zeta(1+\alpha+\beta)$.   In addition,
$T^-:=\{-t:t\in T\}$, $A=S+\overline{S}$, $B=T+\overline{T}$ and
\begin{eqnarray}\label{eq:HHHH}
\mathcal H_{S,T} (W_r)= H_{\zeta;S,T}(W_r)-\sum_p
H_{p,1;S,T}(W_r)+\sum_p H_{p,2;S,T}(W_r).
\end{eqnarray}
Further, we have
\begin{equation}\label{eqn:H}
H_{\zeta,S,T}(W)=\left\{\begin{array}{ll} \sum_{\hat \alpha\in
S}\frac{\zeta'}{\zeta}(1+\alpha-\hat{\alpha})-\sum_{\hat\beta\in
T} \frac{\zeta'}{\zeta}(1+\alpha +\hat \beta) &\mbox{ if
$W=\{\alpha\}\subset \overline{S}$}
   \\
\sum_{\hat\beta\in T}\frac{\zeta'}{\zeta}(1+\beta-\hat
\beta)-\sum_{\hat\alpha\in S} \frac{\zeta'}{\zeta}
(1+\beta+\hat\alpha) &\mbox{ if  $W=\{\beta\}\subset \overline{T}$}\\
\left(\frac{\zeta'}{\zeta}\right)'(1+\alpha+\beta) & \mbox{ if
$W=\{\alpha,\beta\}$ with $
{\alpha \in \overline{S}, \atop \beta\in \overline{T}}$}\\
0&\mbox{ otherwise};
\end{array}
\right.
\end{equation}
\begin{equation}\label{eqn:HH}
H_{p,1;S,T}(W)=\left\{\begin{array}{ll} \sum_{\hat \alpha\in
S}\frac{z_p'}{z_p}(1+\alpha-\hat{\alpha})-\sum_{\hat\beta\in T}
\frac{z_p'}{z_p} (1+\alpha+\hat \beta) &\mbox{ if
$W=\{\alpha\}\subset \overline{S}$}
   \\
\sum_{\hat\beta\in T}\frac{z_p'}{z_p}(1+\beta-\hat
\beta)-\sum_{\hat\alpha\in S} \frac{z_p'}{z_p}
(1+\beta+\hat\alpha) &\mbox{ if  $W=\{\beta\}\subset \overline{T}$}\\
\left(\frac{z_p'}{z_p}\right)'(1+\alpha+\beta) & \mbox{ if
$W=\{\alpha,\beta\}$ with $
{\alpha \in \overline{S}, \atop \beta\in \overline{T}}$}\\
0&\mbox{ otherwise};
\end{array}
\right.
\end{equation}
and
\begin{eqnarray}
   H_{p,2,S,T}(W)
&&  =\sum_{W =\sum_{j=1}^J X_j}(-1)^{J-1}(J-1)!\prod_{j=1}^J
c_{S,T}(X_j)
\end{eqnarray}
with
\begin{eqnarray}
c_{S,T}(X) :=  \frac{\int_0^1  \mathcal
A_{p,\theta}(S,T)\prod_{\alpha\in \overline{S}\cap
X}\frac{z_{p,-\theta}'}{z_{p,-\theta}}(\frac
12+\alpha)\prod_{\beta\in \overline{T}\cap X}
\frac{z_{p,\theta}'}{z_{p,\theta}}(\frac 12+ \beta)
   ~d\theta}{\int_0^1 \mathcal A_{p,\theta}(S,T)~d\theta}
\end{eqnarray}
and the notation
\begin{eqnarray}
\mathcal A_{p,\theta}(S,T):=\mathcal A_{p,\theta}(T^-,S^-;S,T),
\end{eqnarray}
with $A_{p,\theta}(A,B;C,D)$ as in Conjecture
\ref{conj:ratiozeta}.
  \end{theorem}

To prove this we want to differentiate $\mathcal R_{\zeta,t}$ with
respect to all of the $\alpha\in A$ and $\beta\in B$ and then
replace each $\gamma$ by an $\alpha$ and each $\delta$ by a
$\beta$. In what follows $A$ and $C$ have the same cardinality, as
do $B$ and $D$. Thus, after differentiation, when we want to set
the $\gamma$ equal to the $\alpha$ in some order, and the $\delta$
equal to the $\beta$ in some order, we can abbreviate this by
$C=A$ and $D=B$. With the definition of $J_{\zeta,I}(A;B;U)$ as in
(\ref{eq:defJzeta}), and using Corollary \ref{cor:chis}, we have
\begin{eqnarray}
&&J_{\zeta,I}(A;B;U)\\&&\qquad=\int_I\prod_{\alpha\in A \atop
\beta \in B}\frac{d}{d\alpha}\frac{d}{d\beta} \mathcal
R_{\zeta,t}(A,B;C,D)\bigg|_{C=A\atop D=B}\prod_{\mu\in
U}\Big(-\frac{\chi'}{\chi}(1/2+it+\mu)\Big)~dt+O(|I|^{1/2+\epsilon}).\nonumber
\end{eqnarray}

The situation is much as in the random matrix theory case, except
that now we have to understand how to include the arithmetical
factor $\mathcal A_\zeta$. For a start, we can differentiate with
respect to the $\hat\alpha\in S$ and $\hat \beta\in T$ as before.
$C_{\overline{S}}$ and $D_{\overline{T}}$ are defined as in the
proof of Theorem \ref{theo:J}. We have
\begin{eqnarray}\label{eq:halfderiv}
&& J_{\zeta,I}(A;B;U) = \int_I \prod_{\mu\in
U}\Big(-\frac{\chi'}{\chi}(1/2+it+\mu)\Big)\sum_{S\subset
A,T\subset B\atop |S|=|T|}X_t(S,T)
\frac{Z_\zeta(S,T)Z_\zeta(S^-,T^-)}{{Z_\zeta}'(S,S^-){Z_\zeta}'(T,T^-)}
\\&& \qquad \times  \nonumber
\prod_{\alpha\in \overline{S}\atop \beta\in
\overline{T}}\frac{d}{d\alpha}\frac{d}{d\beta}\Bigg(
\frac{Z_\zeta(\overline{S},\overline{T})Z_\zeta(\overline{S},S^-)Z_\zeta(\overline{T},T^-)Z(C,D)}
{Z_\zeta(S,T)Z_\zeta(C_{\overline{S}},S^-)Z_\zeta(D_{\overline{T}},T^-)
Z_\zeta(\overline{S},D)Z_\zeta(\overline{T},C)
  }\nonumber \\
  &&\qquad\qquad\qquad\qquad\qquad\times\mathcal{A}_\zeta(\overline{S}+ T^-,\overline{T}+ S^-;C,D)
  \Bigg)\bigg|_{C=A\atop D=B}~dt +O(|I|^{1/2+\epsilon}) .\nonumber\end{eqnarray}

In anticipation of applying Lemma \ref{lemma:diff}, we note that a
brief calculation shows that
\begin{eqnarray}  \label{eqn:Aid} \nonumber
&&\frac{Z_\zeta(\overline{S},\overline{T})Z_\zeta(\overline{S},S^-)
Z_\zeta(\overline{T},T^-)Z(C,D)}
{Z_\zeta(S,T)Z_\zeta(C_{\overline{S}},S^-)Z_\zeta(D_{\overline{T}},T^-)
Z_\zeta(\overline{S},D)Z_\zeta(\overline{T},C)
  }\mathcal{A}_\zeta(\overline{S}+ T^-,\overline{T}+ S^-;C,D)\bigg|_{C=A\atop
  D=B}\\&&  \qquad \qquad
=\mathcal{A}_\zeta(T^-,S^-;S,T).
\end{eqnarray}

The remainder of the proof of Theorem \ref{theo:Jzeta} consists of
applying Lemma \ref{lemma:diff} with (keeping just the factors
from the big brackets in (\ref{eq:halfderiv}) that depend on
$\overline{S}$ and $\overline{T}$)
\begin{eqnarray}
H&=&\log Z_\zeta(\overline{S},\overline{T})+\log
Z_\zeta(\overline{S},S^-)+\log Z_\zeta(\overline{T},T^-)-\log
Z_\zeta(\overline{S},D) -\log Z_\zeta(\overline{T},C) \nonumber
\\
&&+\sum_p \big(\log Z_p(\overline{S},\overline{T})+\log
Z_p(\overline{S},S^-)+\log Z_p(\overline{T},T^-)-\log
Z_p(\overline{S},D) -\log Z_p(\overline{T},C) \\
&&\qquad\qquad+ \log \int_0^1 \mathcal
A_{p,\theta}(\overline{S}+T^-,\overline{T}+S^{-};C,D)d \theta
.\nonumber
\end{eqnarray}
Note the exponent $-1$ in the definition of $Z_p(A,B)$ in
Conjecture \ref{conj:ratiozeta} which accounts for the minus sign
in front of $H_{1,p,S,T}(W_r)$ in the definition in
(\ref{eq:HHHH}) of $H_{S,T}(W_r)$.

Now it just remains to prove:
\begin{lemma} \label{lem:lemma3}
Let $W\subset \overline{S}+ \overline{T}$. Then
\begin{eqnarray}\label{eq:hp2}
  \nonumber H_{p,2,S,T}(W)&:=& \left.\prod_{w\in W}\frac{d}{dw} \log\left(
\int_0^1\mathcal A_{p,\theta}(\overline{S}+ T^-,\overline{T}+ S^-
;C,D)~d\theta\right)\right|_{C=\overline{S}+ S
\atop D=\overline{T}+ T}\\
&  =&\sum_{W =\sum_{j=1}^J X_j}(-1)^{J-1}(J-1)!\prod_{j=1}^J
c_{S,T}(X_j);
\end{eqnarray}
here
\begin{eqnarray}\label{eq:cST}
c_{S,T}(X) :=  \frac{\int_0^1  \mathcal
A_{p,\theta}(S,T)\prod_{\alpha\in \overline{S}\cap
X}\frac{z_{p,-\theta}'}{z_{p,-\theta}}(\frac
12+\alpha)\prod_{\beta\in \overline{T}\cap X}
\frac{z_{p,\theta}'}{z_{p,\theta}}(\frac 12+ \beta)
   ~d\theta}{\int_0^1 \mathcal A_{p,\theta}(S,T)~d\theta}
\end{eqnarray}
where we have adopted the notation
\begin{eqnarray}
\mathcal A_{p,\theta}(S,T):=\mathcal A_{p,\theta}(T^-,S^-;S,T).
\end{eqnarray}
Further,   we have, for $\alpha^*\in S$, $\beta^*\in T$,
$S'=S-\{\alpha^*\}$, $T'=T-\{\beta^*\}$ and $W\subset
\overline{S}+\overline{T}$,
   \begin{eqnarray}&&\label{eq:cSTderiv}
\nonumber\frac{d}{d\alpha^*}c_{S,T}(W)\bigg|_{\alpha^*=-\beta^*}=-c_{S',T'}(W+\{\alpha^*\})|_{\alpha^*=-\beta^*}
-c_{S',T'}(W+
\{\beta^*\})\\
&&\quad \qquad
+c_{S',T'}(W)c_{S',T'}(\{\alpha^*\})|_{\alpha^*=-\beta^*}+c_{S',T'}(W)c_{S',T'}(\{\beta^*\})
\end{eqnarray}
and
\begin{eqnarray} \label{eqn:dHp2}
 && \frac{d}{d\alpha^*}H_{p,2;S,T}(W)\bigg|_{\alpha^*=-\beta^*}= \nonumber \\
  &&\qquad\qquad-H_{p,2,S',T'}(W+
\{\alpha^*\})|_{\alpha^*=-\beta^*}-H_{p,2,S',T'}(W+ \{\beta^*\}).
\end{eqnarray}
\end{lemma}

\begin{proof}
The proof is simple, involving only differentiation.  The first
line, (\ref{eq:hp2}), follows from logarithmic differentiation of
the integral of $\mathcal A_{p,\theta}$, where each variable $w\in
W$ appears in just one place.  The equation (\ref{eq:cSTderiv})
also arises immediately by following the rules of differentiation.
Note that $\alpha^*$ appears in $\mathcal A_{p,\theta}(S,T)$ in
the numerator of (\ref{eq:Aptheta}) in a factor
$z_{p,\theta}(\frac 12 -\alpha^*)$ and in the denominator in a
factor of type $z_{p,-\theta}(\frac 12 +\alpha^*)$ and
$\alpha^*\notin W$.  The first term in (\ref{eq:cSTderiv}) comes
from the $z_{p,-\theta}(\frac 12 +\alpha^*)$ factor in the
numerator of (\ref{eq:cST}), the second term in
(\ref{eq:cSTderiv}) from the $z_{p,\theta}(\frac 12 -\alpha^*)$
and the final two terms in (\ref{eq:cSTderiv}) from the integral
in the denominator of (\ref{eq:cST}).  To obtain (\ref{eqn:dHp2})
note that the $\alpha^*$ appears in each factor of $c_{S,T}(X_j)$,
$j=1,\ldots,J$, in (\ref{eq:hp2}). Using the product rule on each
term in (\ref{eq:hp2}) we differentiate each $c_{s,T}(X_j)$ in
turn and sum the results. Note that
$c_{S,T}(X_j)\big|_{\alpha^*=-\beta^*}= c_{S',T'}(X_j)$. Using
(\ref{eq:cSTderiv}) and careful combinatorial accounting
(\ref{eqn:dHp2}) can be obtained.
\end{proof}

\begin{remark} We have allowed here some loose use of notation in
writing $H_{p,2,S',T'}(W+\{\alpha^*\})$.  The lemma starts out by
defining $H_{p,2,S,T}(W)$ where
$W\subset\overline{S}+\overline{T}$, but of course
$W+\{\alpha^*\}\notin \overline{S}+\overline{T}$ when $\alpha^*\in
S$. However to understand the notation
$H_{p,2,S',T'}(W+\{\alpha^*\})$ simply replace $S$ with $S'$ and
$T$ with $T'$ in the definition of $H_{p,2,S,T}$ and replace
$\overline{S}$ with $A-S'$ and $\overline{T}$ with $B-T'$.  A
similar comment applies to $c_{S',T'}(W+\{\alpha^*\})$ in
(\ref{eq:cSTderiv}).
\end{remark}

\subsection{Residue identity revisited}

Using Theorem \ref{theo:Jzeta},  which is an application of Lemma
\ref{lemma:diff} along much the same lines as Theorem
\ref{theo:J}, we now prove the analogue for $\zeta$ of Lemma
\ref{lem:residue}.
\begin{lemma} \label{lem:residuezeta} Suppose that $\alpha^*\in A$ and $\beta^*\in B$.  Let $A'=A-\{\alpha^*\}$
and $B'=B-\{\beta^*\}$ and $\ell=\log \frac{t}{2\pi}$. Then
$J^*_{\zeta,t}(A;B;U)$ (defined in Theorem \ref{theo:Jzeta}) has a
simple pole at $\alpha^*=-\beta^*$ with
\begin{eqnarray}
&&\operatornamewithlimits{Res}_{\alpha^*=-\beta^*}
J^*_{\zeta,t}(A;B;U) =-\frac{\chi'}{\chi}(s-\beta^*)
J^*_{\zeta,t}(A';B';U)\nonumber
\\
&&\qquad\qquad\qquad\qquad\qquad\qquad+J^*_{\zeta,t}(A';B;U)+J^*_{\zeta,t}(A'+\{-\beta^*\};B';U).
\end{eqnarray}
\end{lemma}

\begin{proof}  We use Lemma \ref{lem:residue}.
First we remember the convention that $A=S+\overline{S}$ and
$B=T+\overline{T}$ and we write Theorem \ref{theo:Jzeta} as
\begin{eqnarray}
J^*_{\zeta,t}(A,B;U)=\prod_{\mu\in U}\frac{\chi'}{\chi}(1/2+it+\mu)\sum_{{S\subset A\atop T\subset B}\atop
|S|=|T|} D_{\zeta; S,T}(\overline{S},\overline{T})
\end{eqnarray}
where, with the abbreviation $\mathcal A(S,T):=\mathcal
A_{\zeta}(T^-,S^-;S,T)$, we have
\begin{eqnarray}
  D_{\zeta; S,T}(\overline{S},\overline{T})=Q_\zeta(S,T)\mathcal{A}(S,T)\sum_{\overline{S}+ \overline{T} =\sum W_r}\prod_{r=1}^R
\mathcal H_{S,T}(W_r)
\end{eqnarray}
and
\begin{eqnarray}
Q_\zeta(S,T):=X_t(S,T)
\frac{Z_\zeta(S,T)Z_\zeta(S^-,T^-)}{{Z_\zeta}^\dagger(S,S^-){Z_\zeta}^\dagger(T,T^-)}.
\end{eqnarray}
Now we let $Q_\zeta(S,T)\mathcal{A}(S,T)$ play the role of
$Q(S,T)$ in the proof of Lemma \ref{lem:residue} and $\mathcal
H_{S,T}(W)$ plays the role of $H_{S,T}(W)$.  Thus we need to prove
the four conditions below, describing the behaviour of the various
components of the formula as $\alpha^*$ approaches $-\beta^*$, and
then the rest of the proof is identical to that of Lemma
\ref{lem:residue}.

\begin{description}

\item[Q1] If $\alpha^*\in \overline{S}$ and $\beta^*\in \overline{T}$, then
$Q_\zeta(S,T)\mathcal{A}(S,T)$ is independent of $\alpha^*$ and
$\beta^*$ and
\begin{eqnarray}
\mathcal H_{S,T}(W)=\left\{ \begin{array}{ll}
\frac{1}{(\alpha^*+\beta^*)^2}+O(1)
& \mbox{ if $W=\{\alpha^*,\beta^*\}$ }\\
$O(1)$ & \mbox{ otherwise }
\end{array} \right.
\end{eqnarray}
\item[Q2] If $\alpha^*\in S$ and $\beta^*\in \overline{T}$, then
$Q_\zeta(S,T)\mathcal{A}(S,T)$ is regular when $\alpha^*=-\beta
^*$ and
\begin{eqnarray}
\mathcal H_{S,T}(W)=\left\{ \begin{array}{ll}
\frac{1}{\alpha^*+\beta^*}+O(1)
& \mbox{ if $W=\{\beta^*\}$ }\\
$O(1)$ & \mbox{ otherwise }
\end{array} \right.
\end{eqnarray}
\item[Q3] If $\alpha^*\in \overline{S}$ and $\beta^*\in T$, then
$Q_\zeta(S,T)\mathcal{A}(S,T)$ is regular when $\alpha^*=-\beta^*$
and
\begin{eqnarray}
\mathcal H_{S,T}(W)=\left\{ \begin{array}{ll}
\frac{1}{\alpha^*+\beta^*}+O(1)
& \mbox{ if $W=\{\alpha^*\}$ }\\
$O(1)$ & \mbox{ otherwise }
\end{array} \right.
\end{eqnarray}
\item[Q4] If $\alpha^*\in S$ and $\beta^*\in T$ and
$S'=S-\{\alpha^*\}$ and $T'=T-\{\beta^*\}$, then
$Q_\zeta(S,T)=\big(\frac{-1}{(\alpha^*+\beta^*)^2}+O(1)\big)Q_{\zeta;1}(S,T)$
and
\begin{eqnarray}
&&Q_{\zeta;1}(S,T)\mathcal A(S,T)=Q_\zeta(S',T')\mathcal
A(S',T')\bigg(1\nonumber\\
&& \qquad-(\alpha^*+\beta^*)\big(-\frac{\chi'}{\chi}(s-\beta^*)
+\mathcal H_{S',T'}(\{\alpha^*\})|_{\alpha^*=-\beta^*}+\mathcal
H_{S',T'}(\{\beta^*\})\big)\bigg)  +O(1)
\end{eqnarray}
and
\begin{eqnarray}&&\label{eq:Hexpa}
\mathcal H_{S,T}(W)= \mathcal H_{S',T'}(W)-
(\alpha^*+\beta^*)(\mathcal
H_{S',T'}(W+\{\alpha^*\})|_{\alpha^*=-\beta^*}+\mathcal
H_{S',T'}(W+\{\beta^*\}))\\
&&\qquad \qquad \qquad \qquad +O(|\alpha^*+\beta^*|^2) .\nonumber
\end{eqnarray}
\end{description}

We have just to prove these four conditions to complete the proof.
We start with the first case where $\alpha^*\notin S$,
$\beta^*\notin T$. Then $\alpha^*\in \overline{S}$ and $\beta^*\in
\overline{T}$. The terms $H_{p,1}$ and $H_{p,2}$ have no poles
because of the conditions on the real parts of $\alpha$ and
$\beta$.  The only polar term from $\alpha^*=-\beta^*$ arises from
a situation when one of the partition parts is
$W_r=\{\alpha^*,\beta^*\}$ and there is a pole from
$H_{\zeta;S,T}(W_r)=\left(\frac{\zeta'}{\zeta}\right)'(1+\alpha^*+\beta^*)$.
Since $\left(\frac{\zeta'}{\zeta}\right)'(1+x)=1/x^2+O(1)$ and
$Q_\zeta(S,T)\mathcal{A}(S,T)$ is clearly independent of
$\alpha^*$ and $\beta^*$, the first condition is satisfied.

Next, suppose that $\alpha^*\in S$ and $\beta^*\notin T$. The only
pole in $D_{\zeta;S,T}(\overline{S},\overline{T})$ occurs in the
product of the $H$ for $H_{\zeta;S,T}(W_r)$ when $W_r=\{\beta^*\}$.
 We have
\begin{eqnarray}\label{eqn:innotinz}
H_{\zeta;S,T}(\{\beta^*\})=\sum_{\hat\beta \in
T}\frac{\zeta'}{\zeta}(1+\hat \beta-\beta^*)- \sum_{\hat \alpha\in
S}\frac{\zeta'}{\zeta}(1+\beta^*+\hat \alpha)
\end{eqnarray}
for which, when $\hat \alpha=\alpha^*$, the term
$-\frac{\zeta'}{\zeta}(1+\beta^*+\alpha^*)$ has a simple pole at $
\alpha^*=-\beta^*$ with residue 1.  $Q_\zeta(S,T)\mathcal{A}(S,T)$
depends on $\alpha^*$  and not $\beta^*$, so it is regular when
$\alpha^*=- \beta^*$.

Similarly, when $\alpha^*\notin S$ and $\beta^*\in T$, the only
pole in the product of the $H$ occurs for
$H_{\zeta;S,T}(\{\alpha^*\})$. We have
\begin{eqnarray}\label{eqn:innotinz}
H_{\zeta;S,T}(\{\alpha^*\})=\sum_{\hat\alpha \in
S}\frac{\zeta'}{\zeta}(1+\hat \alpha-\alpha^*)- \sum_{\hat
\beta\in T}\frac{\zeta'}{\zeta}(1+\alpha^*+\hat \beta)
\end{eqnarray}
for which, when $\hat \beta =\beta^*$, the term
$-\frac{\zeta'}{\zeta}(1+\alpha^*+\beta^*)$ has a simple pole at
$\alpha^*=-\beta^*$ with residue 1.

  Finally, we consider the case
$\alpha^*\in S$ and $\beta^*\in T$. Let $S'=S-\{\alpha^*\}$ and
$T'=T-\{\beta^*\}$. We have
\begin{eqnarray}
Q_\zeta(S,T)=\zeta(1+\alpha^*+\beta^*)\zeta(1-\alpha^*-\beta^*)
Q_{\zeta;1}(S,T)
\end{eqnarray}
where
\begin{eqnarray}
Q_{\zeta;1}(S,T)&=&Q_\zeta(S',T')X_t(\{\alpha^*\},\{\beta^*\})\nonumber
\\
&&\qquad \times \frac{\prod_{\hat \beta\in
T'}\zeta(1+\alpha^*+\hat \beta) \zeta(1-\alpha^*-\hat
\beta)\prod_{\hat \alpha\in S'} \zeta(1+\hat
\alpha+\beta^*)\zeta(1-\hat \alpha-\beta^*)}{\prod_{\hat \alpha\in
S'} \zeta(1+\alpha^*-\hat \alpha)\zeta(1+\hat
\alpha-\alpha^*)\prod_{\hat \beta\in T'} \zeta(1+\beta^*-\hat
\beta)\zeta(1+\hat \beta-\beta^*)}.
\end{eqnarray}

Note that $\zeta(1+\alpha^*+\beta^*)\zeta(1-\alpha^*-\beta^*)
=\frac{-1}{(\alpha^*+\beta^*)^2}+O(1)$. Also, remembering that
$\chi(s-\beta^*)\chi(1-s+\beta^*)=1$,
\begin{eqnarray}
Q_{\zeta;1}(S,T)\big|_{\alpha^*=-\beta^*}=Q_\zeta(S',T'),
\end{eqnarray}
which gives us an expansion for $Q_{\zeta;1}(S,T)$ in the
neighborhood of $\alpha^*=-\beta^*$:
\begin{eqnarray} \label{eqn:Qz} \nonumber &&
Q_{\zeta;1}(S,T)=
Q_\zeta(S',T')(1+\frac{\chi'}{\chi}(s-\beta^*)(\alpha^*+\beta^*)+O(|\alpha^*+\beta^*|^2)\\
&&\qquad \times \bigg(1+(\alpha^*+\beta^*)\bigg( \sum_{\hat\alpha
\in S'}\Big( \frac{\zeta'}{\zeta}(1+\hat \alpha+\beta^*)-
\frac{\zeta'}{\zeta}(1-\beta^*-\hat\alpha)\Big)\nonumber\\
&&\qquad \qquad +\sum_{\hat \beta\in T'}
\Big(\frac{\zeta'}{\zeta}(1-\beta^*+\hat
\beta)-\frac{\zeta'}{\zeta}(1+\beta^*-\hat \beta)\Big)\bigg)
  +O(|\alpha^*+\beta^*|^2)\bigg)\nonumber\\
  &&\qquad =
Q_\zeta(S',T')\Big(1-(\alpha^*+\beta^*)\big(-\frac{\chi'}{\chi}(s-\beta^*)+H_{\zeta;S',T'}(\{\alpha^*\})|_{\alpha^*=-\beta^*}\nonumber
\\
&&\qquad\qquad\qquad+
  H_{\zeta;S',T'}(\{\beta^*\})\big)+O(|\alpha^*+\beta^*|^2)\Big).
\end{eqnarray}
Since $\left. \mathcal A(S,T)\right|_{\alpha^*=-\beta^*}=\mathcal
A (S',T')$, we have the expansion around $\alpha^*=-\beta^*$:
\begin{eqnarray} \label{eqn:Az} \nonumber &&
\mathcal A(S,T)= \mathcal A(S',T')
  \bigg(1+(\alpha^*+\beta^*)\sum_p\bigg( \sum_{\hat\alpha \in S'}\big(- \frac{z_p'}{z_p}(1+\hat
  \alpha+\beta^*)+
\frac{z_p'}{z_p}(1-\beta^*-\hat\alpha)\big)\\
&&\qquad \qquad +\sum_{\hat \beta\in T'}
\big(-\frac{z_p'}{z_p}(1-\beta^*+\hat
\beta)+\frac{z_p'}{z_p}(1+\beta^*-\hat \beta)\big)\nonumber\\
&&\qquad -\frac{\int_0^1\mathcal A_{p,\theta}(S',T')\big(\frac
{z_{p,-\theta}'}{z_{p,-\theta}}(\frac 12
-\beta^*)+\frac{z_{p,\theta}'}{z_{p,\theta}}(\frac 12
+\beta^*)\big)~d\theta}{\int_0^1\mathcal
A_{p,\theta}(S',T')~d\theta}\bigg)
  +O(|\alpha^*+\beta^*|^2)\bigg)\nonumber\\
  &&=\mathcal A(S',T')
\bigg(1+(\alpha^*+\beta^*)\sum_p\big(H_{p,1;S',T'}(\{\alpha^*\})|_{\alpha^*=-\beta^*}+
  H_{p,1;S',T'}(\{\beta^*\})\nonumber\\
  &&\qquad \qquad \qquad \quad
  -H_{p,2;S',T'}(\{\alpha^*\})|_{\alpha^*=-\beta^*}-
  H_{p,2;S',T'}(\{\beta^*\})\big)\bigg) +O(|\alpha^*+\beta^*|)^2,
\end{eqnarray}
where the first line is a result of differentiating
(\ref{eq:Azeta}), and the second line from the definitions of
$H_{p,1;S,T}(W)$ (in Theorem \ref{theo:Jzeta})  and
$H_{p,2;S,T}(W)$ (in (\ref{eq:hp2})).

Thus we have
\begin{eqnarray}
&&Q_{\zeta;1}(S,T)\mathcal A(S,T)=Q_\zeta(S',T')\mathcal
A(S',T')\bigg(1\nonumber\\
&&\qquad -(\alpha^*+\beta^*)\big(-\frac{\chi'}{\chi}(s-\beta^*) +\mathcal
H_{S',T'}(\{\alpha^*\})|_{\alpha^*=-\beta^*}+\mathcal
H_{S',T'}(\{\beta^*\})\big)\bigg)  +O(1)
\end{eqnarray}
as $\alpha^*\to -\beta^*$.

Now we obtain an expansion for the product of $H$ term. By the
definition of $H_\zeta$ in Theorem \ref{theo:Jzeta} we have that
\begin{eqnarray}
H_{\zeta;S,T}(\{\alpha\})&=&\sum_{\hat{\alpha}\in
S}\frac{\zeta'}{\zeta}(1+\alpha-\hat \alpha)-\sum_{\hat \beta\in
T}
\frac{\zeta'}{\zeta}(1+\alpha+\hat \beta)\nonumber\\
&=&  H_{\zeta;S',T'}(\{\alpha\})+
\frac{\zeta'}{\zeta}(1+\alpha-\alpha^*)-
\frac{\zeta'}{\zeta}(1+\alpha+ \beta^*);
\end{eqnarray}
\begin{eqnarray}
H_{\zeta;S,T}(\{\beta\})&=&\sum_{\hat \beta \in
T}\frac{\zeta'}{\zeta}(1+\beta-\hat \beta )-
\sum_{\hat \alpha \in S} \frac{\zeta'}{\zeta}(1+\beta+\hat \alpha) \nonumber\\
&=& H_{\zeta;S',T'}(\{\beta\})+
\frac{\zeta'}{\zeta}(1+\beta-\beta^*)-
\frac{\zeta'}{\zeta}(1+\beta+ \alpha^*);
\end{eqnarray}
and
\begin{eqnarray}
H_{\zeta;S,T}(\{\alpha,\beta\})=\left(\frac{\zeta'}{\zeta}\right)'(1+\alpha
+\beta)=H_{\zeta;S',T'}(\{\alpha,\beta\}).
\end{eqnarray}
Thus,
\begin{eqnarray} \label{eqn:HzST}
  H_{\zeta;S,T}(W)\bigg|_{\alpha^*=-\beta^*}= H_{\zeta;S',T'}(W)
\end{eqnarray}
and
\begin{eqnarray} \label{eqn:dHzeta}
\frac{d}{d\alpha^*}H_{\zeta;S,T}(W)\big|_{\alpha^*=-\beta^*}&=&
\left\{\begin{array}{ll}
-\left(\frac{\zeta'}{\zeta}\right)'(1+\alpha+\beta^*) &\mbox{ if $W=\{\alpha\}\subset \overline{S}$}\\
   -\left(\frac{\zeta'}{\zeta}\right)'(1+\beta-\beta^*)&\mbox{ if $W=\{\beta\}\subset \overline{T}$}\\
0 &\mbox{ otherwise}
\end{array}\right.\\
&=& -H_{\zeta,S',T'}(W+
\{\alpha^*\})|_{\alpha^*=-\beta^*}-H_{\zeta,S',T'}(W+
\{\beta^*\}).\nonumber
\end{eqnarray}
In exactly the same way,
\begin{eqnarray} \label{eqn:Hp1ST}
  H_{p,1;S,T}(W)\bigg|_{\alpha^*=-\beta^*}= H_{p,1;S',T'}(W)
\end{eqnarray}
and
\begin{eqnarray} \label{eqn:dHp1}
&&\frac{d}{d\alpha^*}H_{p,1;S,T}(W)\big|_{\alpha^*=-\beta^*}\nonumber
\\
&&\qquad\qquad= -H_{p,1,S',T'}(W+
\{\alpha^*\})|_{\alpha^*=-\beta^*}-H_{p,1,S',T'}(W+ \{\beta^*\}).
\end{eqnarray}

Also, by Lemma \ref{lem:lemma3} we have
\begin{eqnarray} \label{eqn:Hp2ST}
  H_{p,2;S,T}(W)\bigg|_{\alpha^*=-\beta^*}= H_{p,2;S',T'}(W)
\end{eqnarray}
and
\begin{eqnarray}\label{eqn:dHp2a}
&&
\frac{d}{d\alpha^*}H_{p,2;S,T}(W)\bigg|_{\alpha^*=-\beta^*}\nonumber
\\
&&\qquad\qquad= -H_{p,2,S',T'}(W+
\{\alpha^*\})|_{\alpha^*=-\beta^*}-H_{p,2,S',T'}(W+ \{\beta^*\}).
\end{eqnarray}

Combining these results we have exactly equation (\ref{eq:Hexpa}).

The  rest of the proof proceeds exactly as before.
  \end{proof}

  \subsection{$n$-correlation via the ratios conjecture}

Now we proceed to $n$-correlation.  Let $f$ satisfy the conditions
\begin{eqnarray} \label{eq:fconditions}
&&f(x_1,\ldots,x_n) \text{ is holomorphic for }  |\Im x_j| <2, \text{ with } j=1,\ldots,n,\\
&& \text{is translation invariant, ie. } f(x_1+t,\ldots,x_n+t)=f(x_1,\ldots,x_n) \nonumber\\
 && \text{and satisfies } f(0,x_2,\ldots,x_n)\ll
1/(1+|x_2|^2+\cdots+|x_n|^2) \text{ as } |x_j| \to
\infty,\nonumber \\
&&\qquad\qquad\qquad\qquad\qquad\qquad\qquad\qquad \text{ with }
j=2,\ldots,n.\nonumber
\end{eqnarray}

\begin{theorem} \label{theo:zetaofftheline}
Let $\mathcal C_-$ denote the path  from $-\delta+ iT$ down to
$-\delta-iT$ and let $\mathcal C_+$ denote the path  from
$\delta-iT$ up to $\delta+iT$ and let $f$ be as in
(\ref{eq:fconditions}). Using the notation $J_{\zeta,t}(A;B;C)$
from Theorem ~\ref{theo:Jzeta},
\begin{eqnarray}\label{eq:zetaofftheline}
&&\sum_{0<\gamma_{j_1},\dots , \gamma_{j_n}\le T}
f(\gamma_{j_1},\dots,\gamma_{j_n})\nonumber
\\
&&\qquad =\frac{1}{(2\pi i)^n} \sum_{K+L+M=
\{1,\dots,n\}}(-1)^{|L|+|M|}  \\
&&\qquad\qquad\qquad \times\int_{\mathcal {C_+}^K} \int_{\mathcal
{C_-}^{L+ M}}\frac{1}{T}\int_{I^*}J_{\zeta,t}(z_K;-z_L;-z_M) ~dt
~f(iz_1,\dots,iz_n)~dz_1\dots ~dz_n\nonumber
\end{eqnarray}
where
  $z_K=\{z_k:k\in K\}$,  $-z_L=\{-z_\ell:\ell\in L\}$ and
  $\int_{\mathcal {C_+}^K} \int_{\mathcal {C_-}^{L+ M}}$
  means that we are integrating all of the variables in $z_K$ along the $\mathcal C_+$
  path  and all of the variables in $z_{L}$ or $z_{M}$ along the $\mathcal C_-$
  path; and  $I^*$ is the interval which has lower endpoint  $\max\{0,-\Im z_1,\ldots,-\Im z_n\}$
and upper endpoint
$\min\{T,T-\Im z_1,\ldots,T-\Im z_n\}$.
  \end{theorem}

\begin{proof}
By Cauchy's theorem we can express the sum over zeros as
\begin{eqnarray}\label{eq:nfold}
&&\sum_{0<  \gamma_1,\dots ,\gamma_n\leq T}
f(\gamma_1,\dots,\gamma_n)\nonumber
\\
&&\qquad\qquad=\frac{1}{(2\pi i)^n} \int_{\mathcal C}\dots
\int_{\mathcal C} f(-iz_1,\dots,-iz_n)\prod_{j=1}^n
\frac{\zeta'}{\zeta}(1/2+z_j)~dz_1\dots dz_n,
\end{eqnarray}
where $\mathcal C$ is a positively oriented contour which encloses
a subinterval of the imaginary axis from zero to $T$. We choose a
specific path $\mathcal C$ to be the positively oriented rectangle
that has vertices $\delta,\delta+iT, -\delta+iT, -\delta$ where
$\delta$ is a small positive number.

Due to the translation invariance of $f$, (\ref{eq:nfold}) equals
\begin{eqnarray}
&&\frac{1}{T}\int_0^T\frac{1}{(2\pi i)^n} \int_{\mathcal C}\dots
\int_{\mathcal C}
f(-iz_1-t,\dots,-iz_n-t)\nonumber \\
&&\qquad\qquad\qquad\qquad\qquad\qquad\qquad\qquad
\times\prod_{j=1}^n
\frac{\zeta'}{\zeta}(1/2+z_j)~dz_1\dots dz_n~dt\nonumber \\
&&=\frac{1}{T}\int_0^T\frac{1}{(2\pi i)^n} \int_{\mathcal
C_{-it}}\dots \int_{\mathcal C_{-it}}
f(-iz_1,\dots,-iz_n)\nonumber
\\
&& \qquad\qquad\qquad\qquad\qquad\qquad\qquad\times\prod_{j=1}^n
\frac{\zeta'}{\zeta}(1/2+it+z_j)~dz_1\dots
dz_n~dt\nonumber \\
&&=\frac{1}{(2\pi i)^n}
\sum_{\epsilon_j\in\{-1,+1\}}\int_{\mathcal C_{\epsilon_n}}\dots
\int_{\mathcal C_{\epsilon_1}}\frac{1}{T}\int_{I^*}
f(-iz_1,\dots,-iz_n)\nonumber
\\
&& \qquad\qquad\qquad\qquad\qquad\qquad\qquad\times\prod_{j=1}^n
\frac{\zeta'}{\zeta}(1/2+it+z_j)~dt~dz_1\dots
dz_n+O(T^{\epsilon})\label{eq:qxvz}
\end{eqnarray}
where the range of the innermost
integral is the interval $I^*$ which has lower endpoint  $\max\{0,-\Im z_1,\ldots,-\Im z_n\}$
and upper endpoint
$\min\{T,T-\Im z_1,\ldots,T-\Im z_n\}$.
In the second line we made a change of variables $z_j\rightarrow
z_j+it$. The contour $\mathcal C_{-it}$ is $\mathcal C$ shifted
down by $-it$; that is, it runs from $\delta-it,\delta+i(T-t),
-\delta+i(T-t), -\delta-it$.  In progressing to the third line, we
note that the horizontal portions of the contour of integration
can be chosen so that the integral along them is $O(T^{\epsilon})$
(following the identical argument to Davenport \cite{kn:dav80},
page 108), so we concentrate on the vertical sides of the
contours.  When we now exchange the order of integration to move
the $t$ integral to the inside, the integration over
$z_1,\ldots,z_n$ becomes the sum of $2^n$ integrals, each on one
of the contours $\mathcal C_+$ or $\mathcal C_-$ defined in
Theorem \ref{theo:zetaofftheline}.

\begin{remark}The main integral is of size $\approx T \log ^n T$.  The $T$
is a result of $\int_{[-T,T]^n} f \approx T$; the power of the log
comes from the moment of the logarithmic derivative and will
become clear from the examples at the end of the paper.
\end{remark}


For each variable $z_j$ in (\ref{eq:qxvz}) which is on $\mathcal
C_-$ we use the functional equation
\begin{equation}
\label{eq:logderivfe}\frac{\zeta'}{\zeta}(s)=\frac{\chi'}{\chi}(s)-\frac{\zeta'}{\zeta}(1-s)
\end{equation}
to replace $\frac{\zeta'}{\zeta}(s+z_j)$, where $s=1/2+it$.
 In this
way we find that (\ref{eq:qxvz}) equals
\begin{eqnarray}
&& \frac{1}{(2\pi i)^n}\sum_{\epsilon_j\in\{-1,+1\}}\int_{\mathcal
C_{\epsilon_n}}\dots \int_{\mathcal
C_{\epsilon_1}}\frac{1}{T}\int_{I^*} \; \prod_{j=1}^n
\left(\frac{1-\epsilon_j}{2}\frac{\chi'}{\chi}(s+z_j)+\epsilon_j
\frac{\zeta'} {\zeta}
(1/2+\epsilon_j(it+z_j))\right)\\
&&\qquad \qquad \times f(iz_1,\dots,iz_n)~dt~dz_1\dots
dz_n.\nonumber
\end{eqnarray}

Another way to write this equation is
\begin{eqnarray}&& \frac{1}{(2\pi i)^n}
\sum_{K\subset\{1,\dots,n\}} \prod_{j\in K} \int_{\mathcal
C_+} \prod_{j\notin K}\int_{\mathcal
C_-}\frac{1}{T}\int_{I^*}\frac{\zeta'} {\zeta} (s+z_j)
  \left(\frac{\chi'}{\chi}(s+z_j)-\frac{\zeta'}
{\zeta}
(1-s-z_j)\right)\\
&& \qquad \qquad \times f(iz_1,\dots,iz_n)~dt~dz_1\dots
dz_n.\nonumber
\end{eqnarray}
(Note that the $dz$'s are no longer in order
 but this
should not cause confusion.) The expansion of the product over
$j\notin K$ can be easily expressed as a sum over subsets of $K$.
This yields
\begin{eqnarray}
&&   \frac{1}{(2\pi i)^n} \sum_{K+L+M=\{1,\dots,n\}}(-1)^{|L|+|M|}
\prod_{k\in K} \int_{\mathcal C_+}\prod_{\ell\in L}\int_{\mathcal
C_-}\frac{1}{T}\int_{I^*}\frac{\zeta'} {\zeta} (s+z_k)
  \frac{\zeta'}
{\zeta} (1-s-z_{\ell})
  \\
&& \qquad \qquad \times \prod_{m\in M}\int_{\mathcal
C_-}\Big(-\frac{\chi'}{\chi}(s+z_m)\Big)\;\;
  f(iz_1,\dots,iz_n)~dt~dz_1\dots ~dz_n.\nonumber
\end{eqnarray}

\begin{remark}
We note the asymptotic for $\frac{\chi'}{\chi}$:
\begin{eqnarray}
\label{eq:chichi}
\frac{\chi'}{\chi}(1/2+it)=-\log\frac{|t|}{2\pi}\left(1+O\left(\frac
1 {|t|}\right)\right).
\end{eqnarray}
In some applications $|z_m|$ is small relative to $t$ and it
simplifies the formulae to replace $\frac{\chi'}{\chi}(s+z_m)$
with $-\log \tfrac{|t|}{2\pi}$.   However, here where $z_k$ can be
the same size as $t$ we will not use this approximation.
\end{remark}

We have the statement of Theorem
\ref{theo:zetaofftheline}.

\end{proof}



\subsection{$n$-correlation for the Riemann zeros}

We will now state our main theorem.

\begin{theorem}   Assume the Ratios Conjecture
\ref{conj:ratiozeta}. Let $J_{\zeta,t}^*$ be as defined in Theorem
\ref{theo:Jzeta}. Then \label{theo:mainzeta}
\begin{eqnarray}&&\sum_{0<\gamma_1\neq\cdots\neq\gamma_n\leq T} f(\gamma_1,\dots,\gamma_n)
\nonumber
\\
&&\qquad =\frac{1}{(2\pi )^n}\int_{[-T,T]^n} \frac{1}{T}\int_{I^*}
  \sum_{K+L+M=
\{1,\dots,n\}}J_{\zeta,t}^*(-iz_K;iz_L;iz_M)~dt~\\
&&\qquad\qquad\qquad\qquad\qquad\times f(z_1,\dots,z_n)~dz_1\dots
~dz_n+ O(T^{1/2+\epsilon})\nonumber
\end{eqnarray}
where  $-iz_K=\{-iz_k:k\in K\}$, $iz_L=\{iz_\ell:\ell\in L\}$,  and
$iz_M=\{iz_m:m\in M\}$. Moreover, the integrand has no poles
on the path of integration.
\end{theorem}

The proof is nearly identical to that of Theorem \ref{theo:main}.
The only difference is that some care is needed with regard to
endpoints of intervals when we move each new path of integration
onto the imaginary axis. The (slight) difficulty is with poles
that may lie at the very endpoints; this point did not arise in
the random matrix theory context because of the periodicity of the
integrand. However by extending the paths slightly we can
circumvent this difficulty; an argument like that used to handle
the horizontal segments in the proof of Theorem
\ref{theo:zetaofftheline} will work in this case, too, and
introduces an error term of size only $O(T^{\epsilon})$.

It remains to verify that the integrand in Theorem
\ref{theo:mainzeta} has no poles on the path of integration.  We
have already confirmed in Lemma \ref{lem:residue} that each
$J^*(-i\theta_K;i\theta_L)$ has only a simple pole at
$\theta_k=-\theta_\ell$ for $\theta_k\in \theta_K$ and
$\theta_\ell\in \theta_L$.

We check that
\begin{eqnarray}\label{eq:singsetzeta}
\sum_{K+L+M=\{1,2,\ldots,n\}}J_{\zeta,t}^*(-i\theta_K;i\theta_L;i\theta_M)
\end{eqnarray}
has no pole at $\theta_1=\theta_2$ when the values of the remaining $\theta_j$ are unequal to $\theta_1$ or $\theta_2$.
  A given $J_{\zeta,t}^*(-i\theta_K;i\theta_L;i\theta_M)$ only has
a pole when $\theta_1\in \theta_L$ and $\theta_2 \in \theta_K$, or
vice versa, so
\begin{eqnarray}
&&\operatornamewithlimits{Res}_{\theta_1=\theta_2}
\sum_{K+L+M=\{1,2,\ldots,n\}}
J_{\zeta,t}^*(-i\theta_K;i\theta_L;i\theta_M) \nonumber \\
&&\qquad= \sum_{K+L+M=\{3,\ldots,n\}}
\operatornamewithlimits{Res}_{\theta_1=\theta_2}\Big(J_{\zeta,t}^*(-i\theta_K+\{-i\theta_1\}
;\{i\theta_2\}+i\theta_L;i\theta_M) \nonumber \\
&&\qquad\qquad\qquad\qquad+ J_{\zeta,t}^*(-i\theta_K+\{-i\theta_2\}
;\{i\theta_1\}+i\theta_L;i\theta_M)\Big)=0;
\end{eqnarray}
this is zero because $\operatornamewithlimits{Res}_{s=x}f(s,x)=-
\operatornamewithlimits{Res}_{s=x}f(x,s)$.

Thus if (\ref{eq:singsetzeta}) had a singular set it would be of
complex dimension less than $n-1$ and by standard results in the
theory of several complex variables, this implies that there is no
singular set (see for example \cite{kn:krantz}, Corollary 7.3.2).

Our  proof of $n$-correlation in the case of $\zeta$-zeros is now complete.

\begin{corollary}
By rearranging the integrals, now that we know the integrand has
no singularities, and using the fact that $f$ is translation
invariant we have that the ratios conjecture implies that
\begin{eqnarray}&&\sum_{0<\gamma_1\neq\cdots\neq\gamma_n\leq T} f(\gamma_1,\dots,\gamma_n)
\nonumber
\\
&&\qquad =\frac{1}{T}\int_0^T \frac{1}{(2\pi )^n}\int_{[-T,T]^n} \sum_{K+L+M=
\{1,\dots,n\}}
  J_{\zeta,t}^*(-iz_K+it;iz_L-it;iz_M-it) \\
&&\qquad\qquad\qquad\qquad\qquad\times f(z_1,\dots,z_n)~dz_1\dots
~dz_n~dt+ O(T^{1/2+\epsilon})\nonumber
\end{eqnarray}
\end{corollary}

\section{Examples} In this section we explicitly write out
all of the terms in our expressions for $n$-correlations for RMT
eigenvalues and for $\zeta$-zeros for $2\le n\le 4$. In the case
of the $\zeta$ correlations, we simplify the terms $X_t(S,T)$ and
$\frac{\chi'}{\chi}(s+\alpha)$ using the approximations involving
$\ell=\log \tfrac{t}{2\pi}$ mentioned at (\ref{eq:chiapprox}) and
(\ref{eq:chichi}).

To proceed, we  calculate $J^*(A;B)$ and
$J^*_{\zeta,t}(A;B):=J^*_{\zeta,t}(A;B;U)/\prod_{\mu \in U}
\frac{\chi'}{\chi}(\tfrac{1}{2}+it+\mu)$ for sets $A$ and $B$ with
4 or fewer elements.  $D_{S,T}$ and $D_{\zeta,S,T}$ are defined in
the proofs of Lemma \ref{lem:residue} and Lemma
\ref{lem:residuezeta}, respectively. In the following sections we
first compile $D_{S,T}$, $D_{\zeta,S,T}$, then evaluate $J^*$ and
$J_{\zeta,t}^*$ and then assemble these into the correlation
formulas, $R_{N,n}(x_1,\ldots,x_n)$ and
$R_{\zeta,t,n}(x_1,\ldots,x_n)$. Here we define $R$ via
\begin{eqnarray}
&&\int_{U(N)} \sideset{}{^*}\sum_{1\leq j_1,\ldots,j_n\leq N}
f(\theta_1,\ldots,\theta_n) dX_N\nonumber \\
&&\qquad\qquad = \frac{1}{(2\pi)^n} \int_{[0,2\pi]^n} R_{N,n}
(x_1,\ldots,x_n)f(x_1,\ldots,x_n) dx_1\cdots dx_n
\end{eqnarray}
and
\begin{eqnarray}
&&\sum_{0<\gamma_1\neq \cdots \neq \gamma_n\leq T}
f(\gamma_1,\ldots,\gamma_n) = \frac{1}{(2\pi)^n} \int_{[-T,T]^n}
\frac{1}{T}\int_{I^*} R_{\zeta,t,n}(x_1,\ldots,x_n) dt \nonumber
\\
&&\qquad\qquad\qquad\qquad \times f(x_1,\ldots,x_n)dx_1\cdots
dx_n+O(T^{1/2+\epsilon}),
\end{eqnarray}
(compare these with Theorems \ref{theo:main} and
\ref{theo:mainzeta}).

 Recall that
 \begin{eqnarray}
z(x)=\frac{1}{(1-e^{-x})},
\end{eqnarray}
\begin{eqnarray}
S(x)=S_N(x)=\frac{\sin \frac{Nx}{2}}{\sin \frac x 2},
\end{eqnarray}
\begin{eqnarray}z_p(x):=(1-p^{-x})^{-1},
\end{eqnarray}
 and
\begin{eqnarray}Z_p(A,B)=\prod_{\alpha\in A\atop\beta\in B}
z_p(1+\alpha+\beta)^{-1}.\end{eqnarray}

We also will introduce, as needed below, a number of other
expressions $A(x)$, $B(x)$, etc.; for convenience, these are
listed in Section \ref{sect:auxfunc}.

 \subsection{Pair correlation, RMT}

Suppose that the sets $A$ and $B$ have just one element:
$A=\{a\},B=\{b\}$. We have
\begin{eqnarray}
&&J^*(a;b)= D_{\phi,\phi}+D_{a,b},
\end{eqnarray}
where
\begin{eqnarray}
D_{\phi,\phi}=\left(\frac{z'}{z}\right)'(a+b),
\end{eqnarray}
and
\begin{eqnarray}
D_{a,b}=e^{-N(a+b)}z(a+b)z(-a-b).
\end{eqnarray}
Thus,
\begin{eqnarray} \label{eqn:Jab}
&&J^*(a;b)=
 \left(\frac{z'}{z}\right)'(a+b)
+e^{-N(a+b)}z(a+b)z(-a-b),
\end{eqnarray}

Then the 2-point correlation function is
\begin{eqnarray}
R_{N,2}(u,v) = N^2+J^*(iu;-iv)+J^*(-iu;iv) =\det
\left(\begin{array}{cc}N&S(u-v)\\S(v-u)& N\end{array}\right).
\end{eqnarray}

\subsection{Pair correlation, $\zeta$}
We have
\begin{eqnarray}
\label{eq:I1ab} J^*_{\zeta,t}(a;b)&=& \Big(\frac{\zeta'}{\zeta}\Big)'(1+a+b)- B(a+b) \nonumber\\
&&\qquad+e^{-\ell(a+b)} \zeta(1+a+b)\zeta(1-a-b) A(a+b)
\end{eqnarray}
where $\ell=\log \frac{t}{2\pi}$ and
\begin{eqnarray}
\label{eq:Aprime} A(x)&=&\prod_p
\frac{(1-\tfrac{1}{p^{1+x}})(1-\tfrac{2}{p}+\tfrac{1}{p^{1+x}})}
{(1-\tfrac{1}{p})^2},
\end{eqnarray}
\begin{eqnarray}
\label{eq:Bprime} B(x)&=&\sum_p\left(\frac{\log
p}{p^{1+x}-1}\right)^2,
\end{eqnarray}
and conjecture that
\begin{eqnarray} R_{\zeta,t,2}(u,v)=\ell^2+J^*_\zeta(iu;-iv)+J^*_\zeta(-iu;iv).
\end{eqnarray}
Further, letting
\begin{eqnarray}
P_1(x)=e^{-\ell x}A(x)\zeta(1+x)\zeta(1-x)
\end{eqnarray}
and
\begin{eqnarray}
P_2(x)=\Big(\frac{\zeta'}{\zeta}\Big)'(1+x) - B(x),
\end{eqnarray}
we have
\begin{eqnarray}
J^*_{\zeta,t}(a;b)=P_1(a+b)+P_2(a+b).
\end{eqnarray}

\subsection{Triple correlation, RMT}

In this case we have
\begin{eqnarray}
&&J^*(a;b_1,b_2)=  D_{\phi,\phi}+ D_{a,b_1}+D_{a,b_2},
\end{eqnarray}
\begin{eqnarray}
D_{\phi,\phi}=0,
\end{eqnarray}
\begin{eqnarray}
D_{a,b_1}=e^{-N(a+b_1)}z(a+b_1)z(-a-b_1)
\left(\frac{z'}{z}(b_2-b_1)-\frac{z'}{z}(b_2+a) \right),
\end{eqnarray}
and
\begin{eqnarray}
D_{a,b_2}=e^{-N(a+b_2)}z(a+b_2)z(-a-b_2)
\left(\frac{z'}{z}(b_1-b_2)-\frac{z'}{z}(b_1+a) \right),
\end{eqnarray}

Thus,
\begin{eqnarray}  \label{eqn:Jabb}
J^*(a;b_1,b_2)
  &=&
e^{-N(a+b_1)}z(a+b_1)z(-a-b_1)
\left(\frac{z'}{z}(b_2-b_1)-\frac{z'}{z}(b_2+a) \right)\\
&&  +  \nonumber e^{-N(a+b_2)}z(a+b_2)z(-a-b_2)
\left(\frac{z'}{z}(b_1-b_2)-\frac{z'}{z}(b_1+a) \right) .
\end{eqnarray}

Then
\begin{eqnarray}R_{N,3}(u,v,w)&=&
 N^3 + N\big(J^*(i u; -i v) + J^*(i v; -i u) + J^*(i u; -i
w)\nonumber\\&&\qquad  + J^*(i w; -i u) +
         J^*(i w; -i v) + J^*(i v; -i w)\big)\nonumber\\
         && \qquad + \big(J^*(-i w; i u, i v) +
   J^*(-i v; i w, i u) + J^*(-i u; i w, i v)\nonumber\\
   && \qquad \qquad  + J^*(i u; -i w, -i v) +
   J^*(i v; -i w, -i u) + J^*(i w; -i u, -i v)\big)\\
   &=& \nonumber
   \det\left(\begin{array}{ccc}
   N& S(u - v)& S(u - w)\\ S(v - u)& N& S(v - w)\\ S(w - u)& S(w -
   v)&        N\end{array}\right).
\end{eqnarray}

\subsection{Triple correlation for $\zeta$}

The analogue for $\zeta$ is
\begin{eqnarray}
\label{eqn:Jzabb} &&J^*_{\zeta,t}(a;b_1,b_2)= e^{-\ell(a+b_1)}
A(a+b_1)\zeta(1+a+b_1)
\zeta(1-a-b_1)\nonumber \\
&&\qquad\times\bigg(\frac{\zeta'}{\zeta}(1+b_2-b_1)-\frac
{\zeta'}{\zeta}(1+a+b_2)-B_1(a+b_1,a+b_2)\bigg)\nonumber
\\
&& +e^{-\ell(a+b_2)}A(a+b_2) \zeta(1+a+b_2)
\zeta(1-a-b_2)\nonumber \\
&&\qquad\times\bigg(\frac{\zeta'}{\zeta}(1+b_1-b_2)-\frac
{\zeta'}{\zeta}(1+a+b_1)-B_1(a+b_2,a+b_1)\bigg)\\
&&\qquad\qquad\qquad\qquad+Q(a+b_1,a+b_2),\nonumber
\end{eqnarray}
where
\begin{eqnarray}
\label{eq:Qprime} Q(x,y)&=&-\sum_p \frac{\log^3 p}
{p^{2+x+y}(1-\frac{1}{p^{1+x}})( 1-\frac{1}{p^{1+y}})}
\end{eqnarray}
and
\begin{eqnarray}
\label{eq:Pprime} B_1(x,y)=  \sum_p\frac{\Big(1-\frac{1}{p^x}\Big)
\Big(1-\frac{1}{p^x}-\frac{1}{p^y} +\frac{1}{p^{1+y}}\Big) \log p}
{\Big(1-\frac{1}{p^{1-x+y}}\Big) \Big(1-\frac{1}{p^{1+y}}\Big)
\Big( 1-\frac{2}{p}+\frac{1}{p^{1+x}}\Big)p^{2-x+y}}.\nonumber
\end{eqnarray}
Then we conjecture that
\begin{eqnarray} \nonumber R_{\zeta,t,3}(u,v,w)&=& \ell^3 + \ell\big(J^*_{\zeta,t}(i u;
-i v) + J^*_{\zeta,t}(i v; -i u)+ J^*_{\zeta,t}(i u; -i w)
 \\ \nonumber
&&\qquad + J^*_{\zeta,t}(i w; -i u) +
         J^*_{\zeta,t}(i w; -i v) + J^*_{\zeta,t}(i v; -i w)\big)\\
         && \qquad + \big(J^*_{\zeta,t}(-i w; i u, i v) +
   J^*_{\zeta,t}(-i v; i w, i u) + J^*_{\zeta,t}(-i u; i w, i v)\\ \nonumber
   && \qquad \qquad  + J^*_{\zeta,t}(i u; -i w, -i v) +
   J^*_{\zeta,t}(i v; -i w, -i u) + J^*_{\zeta,t}(i w; -i u, -i v).
\end{eqnarray}

It is convenient to introduce the function
\begin{eqnarray}
P_3(a,b,c)=B_1(a+b,a+c)+\frac{\zeta'}{\zeta}(1+a+c)-\frac{\zeta'}{\zeta}(1+c-b).
\end{eqnarray}
In terms of this, we have
\begin{eqnarray}\nonumber
J^*_{\zeta,t}(\{a\};\{b_1,b_2\})=Q(a+b_1,a+b_2)-P_1(a+b_1)P_3(a,b_1,b_2)-P_1(a+b_2)P_3(a,b_2,b_1).
\end{eqnarray}

\subsection{Quadruple correlation, RMT}

With $A=\{a\},B=\{b_1,b_2,b_3\}$,  we have that
\begin{eqnarray}  \label{eqn:Jabbb}
&& J^*(a;b_1,b_2,b_3)=D_{\phi,\phi}+
D_{a,b_1}+D_{a,b_2}+D_{a,b_3}\nonumber
\\
&&  \qquad = e^{-N(a+b_1)}z(a+b_1)z(-a-b_1)
\left(\frac{z'}{z}(b_2-b_1)-\frac{z'}{z}(b_2+a) \right)
\left(\frac{z'}{z}(b_3-b_1)-\frac{z'}{z}(b_3+a) \right)  \nonumber\\
&& \qquad \quad +    e^{-N(a+b_2)}z(a+b_2)z(-a-b_2)
\left(\frac{z'}{z}(b_1-b_2)-\frac{z'}{z}(b_1+a)
   \right) \left(\frac{z'}{z}(b_3-b_2)-\frac{z'}{z}(b_3+a) \right)
\\
&& \qquad \quad +   \nonumber e^{-N(a+b_3)}z(a+b_3)z(-a-b_3)
\left(\frac{z'}{z}(b_1-b_3)-\frac{z'}{z}(b_1+a)
   \right) \left(\frac{z'}{z}(b_2-b_3)-\frac{z'}{z}(b_2+a) \right)
\end{eqnarray}

With  $A=\{a_1, a_2\},B=\{b_1,b_2\}$,  we have
\begin{eqnarray}
&&J^*(a_1,a_2;b_1,b_2)=  D_{\phi,\phi}+ D_{a_1,b_1}+D_{a_1,b_2}+
D_{a_2,b_1}+D_{a_2,b_2}+D_{a_1,a_2,b_1,b_2}.
\end{eqnarray}
where
\begin{eqnarray}
D_{\phi,\phi}=\left(\frac{z'}{z}\right)'(a_1+b_1)\left(\frac{z'}{z}\right)'(a_2+b_2)
+\left(\frac{z'}{z}\right)'(a_1+b_2)\left(\frac{z'}{z}\right)'(a_2+b_1)
\end{eqnarray}
and
\begin{eqnarray}
D_{a_1,b_1}&=&e^{-N(a_1+b_1)}z(a_1+b_1)z(-a_1-b_1)\\
&&\qquad\nonumber\times\big(H_{\{a_1\},\{b_1\}}(\{a_2\},\{b_2\})
+H_{\{a_1\},\{b_1\}}(\{a_2\})
H_{\{a_1\},\{b_1\}}(\{b_2\})\big)\\
&=& \nonumber e^{-N(a_1+b_1)}z(a_1+b_1)z(-a_1-b_1)
\left(\left(\frac{z'}{z}\right)'(a_2+b_2)\right.
\\
&&\qquad \nonumber\left.+\left(\frac{z'}{z}(a_2-a_1)-\frac{z'}{z}(a_2+b_1)
\right) \left(\frac{z'}{z}(b_2-b_1)-\frac{z'}{z}(b_2+a_1) \right)
\right);
\end{eqnarray}
the other $D_{a_i,b_j}$ are similar. Also,
\begin{eqnarray} &&
D_{a_1,a_2,b_1,b_2}=e^{-N(a_1+a_2+b_1+b_2)}\times \nonumber\\
\nonumber &&\quad \frac{z(a_1+b_1)z(-a_1-b_1)z(a_1+b_2)z(-a_1-b_2)
z(a_2+b_1)z(-a_2-b_1)z(a_2+b_2)z(-a_2-b_2)}
{z(a_1-a_2)z(a_2-a_1)z(b_1-b_2)z(b_2-b_1)}.
\end{eqnarray}
Thus,
\begin{eqnarray}&&  \nonumber
J^*(a_1, a_2; b_1, b_2) :=
   \left(\frac{z'}{z}\right)'(a_1 + b_1) \left(\frac{z'}{z}\right)'(a_2 + b_2)
   + \left(\frac{z'}{z}\right)'(a_1 + b_2) \left(\frac{z'}{z}\right)'(a_2 + b_1)\\
   && \qquad  +  \nonumber
     e^{-N (a_1 + b_1)} z(a_1 + b_1)
       z(-a_1 - b_1) \\
             &&\qquad \qquad \times \bigg(\left(\frac{z'}{z}\right)'(
             a_2 + b_2) + \big(\frac{z'}{z}(a_2 - a_1) - \frac{z'}{z}(a_2 + b_1)\big) \big(\frac{z'}{z}(b_2 - b_1) -
                 \frac{z'}{z}(b_2 + a_1)\big)\bigg) \\
                 &&\qquad  \nonumber +
     e^{-N (a_1 + b_2)} z(a_1 + b_2)
       z(-a_1 - b_2)  \\  \nonumber
             &&\qquad \qquad \times\bigg(\left(\frac{z'}{z}\right)'(
             a_2 + b_1)+ \big(\frac{z'}{z}(a_2 - a_1) - \frac{z'}{z}(a_2 + b_2)\big) \big(\frac{z'}{z}(b_1 - b_2) -
                 \frac{z'}{z}(b_1 + a_1)\big)\bigg)\\
                 &&\qquad  + \nonumber
     e^{-N (a_2 + b_1)} z(a_2 + b_1)
       z(-a_2 - b_1) \\  \nonumber
       && \qquad \qquad \times \bigg(\left(\frac{z'}{z}\right)'(
             a_1 + b_2) + \big(\frac{z'}{z}(a_1 - a_2) - \frac{z'}{z}(a_1 + b_1)\big) \big(\frac{z'}{z}(b_2 - b_1) -
                 \frac{z'}{z}(b_2 + a_2)\big)\bigg)\\
                 &&\qquad +  \nonumber
     e^{-N (a_2 + b_2)} z(a_2 + b_2)
       z(-a_2 - b_2) \\  \nonumber
       &&\qquad \qquad \times \bigg(\left(\frac{z'}{z}\right)'(
             a_1 + b_1) + \big(\frac{z'}{z}(a_1 - a_2) - \frac{z'}{z}(a_1 + b_2)\big) \big(\frac{z'}{z}(b_1 - b_2) -
                 \frac{z'}{z}(b_1 + a_2)\big)\bigg)\\
                 &&\qquad  +  \nonumber
     e^{-N (a_1 + a_2 + b_1 + b_2)}z(a_1 + b_1) z(a_1 + b_2) z(a_2 + b_1) z(a_2 +
     b_2) \\   \nonumber
     &&\qquad \qquad \times \frac{
       z(-a_1 - b_1) z(-a_1 - b_2) z(-a_2 - b_1)
       z(-a_2 - b_2)}{z(a_1 - a_2)z(a_2 - a_1)z(b_1 - b_2)z(b_2 -
       b_1)}.
\end{eqnarray}
Then
\begin{eqnarray}R_{N,4}(u,v,w,y)&=& N^4 + N^2 \big(J^*(i u; -i v) + J^*(i v; -i u) + J^*(i u;
-i w)  \nonumber  \\
&& \qquad \qquad  + J^*(i w; -i u)+
           J^*(i w; -i v) + J^*(i v; -i w)  \nonumber \\
   &&\qquad \qquad + J^*(i y; -i u) + J^*(i u; -i y) +
           J^*(i y; -i v)   \\
           &&\qquad \qquad \nonumber + J^*(i v; -i y) + J^*(i y; -i w) + J^*(i w; -i y)\big) \\
           &&\qquad +  \nonumber
     N \big(J^*(-i w; i u, i v) + J^*(-i v; i w, i u) + J^*(-i u; i w, i v) \\
   &&\qquad \qquad +  \nonumber
           J^*(i u; -i w, -i v) + J^*(i v; -i w, -i u) + J^*(i w; -i u, -i v)\\
   &&\qquad \qquad +   \nonumber
           J^*(-i w; i y, i v) + J^*(-i v; i w, i y)   + J^*(-i y; i w, i v) \\
   &&\qquad  \qquad +  \nonumber
           J^*(i y; -i w, -i v) + J^*(i v; -i w, -i y) + J^*(i w; -i y, -i v) \\
   &&\qquad \qquad +  \nonumber
           J^*(-i w; i u, i y) + J^*(-i y; i w, i u) + J^*(-i u; i w, i y) \\
   &&\qquad \qquad +  \nonumber
           J^*(i u; -i w, -i y)   + J^*(i y; -i w, -i u) + J^*(i w; -i u, -i y) \\
   &&\qquad \qquad +  \nonumber
           J^*(-i y; i u, i v) + J^*(-i v; i y, i u)  + J^*(-i u; i y, i v)\\
   &&\qquad  \qquad +  \nonumber
           J^*(i u; -i y, -i v) + J^*(i v; -i y, -i u) + J^*(i y; -i u, -i v)\big) \\
   &&\qquad   +  \nonumber
     J^*(-i y;i u, i v, i w) + J^*(-i w; i y, i u, i v) +
     J^*(-i v; i y, i u, i w) \\
   &&\qquad  +   \nonumber J^*(-i u; i y, i v, i w)  + J^*(i y; -i u, -i v, -i w) + J^*(i u; -i y, -i v, -i w) \\
   &&\qquad  +  \nonumber
     J^*(i v; -i y, -i u, -i w) + J^*(i w; -i y, -i u, -i v) \\
   &&\qquad +  \nonumber
     J^*(i y, i u; -i v, -i w) + J^*(i y, i v; -i u, -i w) +
     J^*(i y, i w; -i u, -i v) \\
   &&\qquad  +  \nonumber J^*(i u, i v; -i y, -i w) +
     J^*(i u, i w; -i y, -i v) + J^*(i v, i w; -i y, -i u)
\\&=&  \nonumber
   \det\left(\begin{array}{cccc}N& S(u - v)& S(u - w)& S(u - y)\\S(v - u)& N& S(v -
   w)&
         S(v - y)\\ S(w - u)&S(w - v)& N& S(w - y)\\ S(y - u)& S(y -
         v)&
         S(y - w)& N\end{array}\right)
\end{eqnarray}

\subsection{Quadruple correlation, $\zeta$}
We conjecture that
  \begin{eqnarray}  \nonumber  R_{\zeta,t,4}(u,v,w,y)&=& \ell^4 +
  \ell^2 \big(J^*_{\zeta,t}(i u; -i v) + J^*_{\zeta,t}(i v; -i u) + J^*_{\zeta,t}(i u;
-i w)\\  \nonumber
&& \qquad \qquad  + J^*_{\zeta,t}(i w; -i u)+
           J^*_{\zeta,t}(i w; -i v) + J^*_{\zeta,t}(i v; -i w)  \\
   &&\qquad \qquad   + J^*_{\zeta,t}(i y; -i u) + J^*_{\zeta,t}(i u; -i y) +
           J^*_{\zeta,t}(i y; -i v)\\
           &&\qquad \qquad \nonumber  + J^*_{\zeta,t}(i v; -i y) + J^*_{\zeta,t}(i y; -i w) + J^*_{\zeta,t}(i w; -i y)\big) \\
           &&\qquad + \nonumber
     \ell \big(J^*_{\zeta,t}(-i w; i u, i v) + J^*_{\zeta,t}(-i v; i w, i u) + J^*_{\zeta,t}(-i u; i w, i v) \\
   &&\qquad \qquad +  \nonumber
           J^*_{\zeta,t}(i u; -i w, -i v) + J^*_{\zeta,t}(i v; -i w, -i u) + J^*_{\zeta,t}(i w; -i u, -i v)\\
   &&\qquad \qquad +  \nonumber
           J^*_{\zeta,t}(-i w; i y, i v) + J^*_{\zeta,t}(-i v; i w, i y)   + J^*_{\zeta,t}(-i y; i w, i v) \\
   &&\qquad  \qquad +  \nonumber
           J^*_{\zeta,t}(i y; -i w, -i v) + J^*_{\zeta,t}(i v; -i w, -i y) + J^*_{\zeta,t}(i w; -i y, -i v) \\
   &&\qquad \qquad +  \nonumber
           J^*_{\zeta,t}(-i w; i u, i y) + J^*_{\zeta,t}(-i y; i w, i u) + J^*_{\zeta,t}(-i u; i w, i y) \\
   &&\qquad \qquad +  \nonumber
           J^*_{\zeta,t}(i u; -i w, -i y)   + J^*_{\zeta,t}(i y; -i w, -i u) + J^*_{\zeta,t}(i w; -i u, -i y) \\
   &&\qquad \qquad +  \nonumber
           J^*_{\zeta,t}(-i y; i u, i v) + J^*_{\zeta,t}(-i v; i y, i u)  + J^*_{\zeta,t}(-i u; i y, i v)\\
   &&\qquad  \qquad +  \nonumber
           J^*_{\zeta,t}(i u; -i y, -i v) + J^*_{\zeta,t}(i v; -i y, -i u) + J^*_{\zeta,t}(i y; -i u, -i v)\big) \\
   &&\qquad   +   \nonumber
     J^*_{\zeta,t}(-i y;i u, i v, i w) + J^*_{\zeta,t}(-i w; i y, i u, i v) +
     J^*_{\zeta,t}(-i v; i y, i u, i w) \\
   &&\qquad  +   \nonumber J^*_{\zeta,t}(-i u; i y, i v, i w)  + J^*_{\zeta,t}(i y; -i u, -i v, -i w) + J^*_{\zeta,t}(i u; -i y, -i v, -i w) \\
   &&\qquad  +  \nonumber
     J^*_{\zeta,t}(i v; -i y, -i u, -i w) + J^*_{\zeta,t}(i w; -i y, -i u, -i v) \\
   &&\qquad +   \nonumber
     J^*_{\zeta,t}(i y, i u; -i v, -i w) + J^*_{\zeta,t}(i y, i v; -i u, -i w) +
     J^*_{\zeta,t}(i y, i w; -i u, -i v) \\
   &&\qquad  +   \nonumber J^*_{\zeta,t}(i u, i v; -i y, -i w) +
     J^*_{\zeta,t}(i u, i w; -i y, -i v) + J^*_{\zeta,t}(i v, i w; -i y, -i u)
\end{eqnarray}
where the relevant $J^*_{\zeta,t}$ are now described.

We have
\begin{eqnarray}&&
 J^*_{\zeta,t}(\{a\},\{b_1,b_2,b_3\})=-\sum_p\frac{z_p(1+a+b_1)z_p(1+a+b_2)z_p(1+a+b_3)\log^4p }{p^{3+3a+b_1+b_2+b_3}}\\
&&\qquad \qquad \qquad +  \nonumber
W_1(a,b_1;b_2,b_3)+W_1(a,b_2;b_1,b_3)+W_1(a,b_3;b_1,b_2)
\end{eqnarray}
where
\begin{eqnarray}
W_1(a,b_1;b_2,b_3)=P_1(a+b_1)(P_3(a,b_1,b_2)P_3(a,b_1,b_3)-B_2(a,b_1;b_2,b_3)).
\end{eqnarray}
   with
   \begin{eqnarray}
   B_2(a,b_1;b_2,b_3)=\sum_p \frac{(p-1) p^{2 b_1} \left(p^{a+b_1}-1\right) \left(p^{a+b_1}-p\right) \log ^2 p}{\left(-2 p^{a+b_1}+p^{a+b_1+1}+1\right)^2
   \left(p^{b_1}-p^{b_2+1}\right) \left(p^{b_1}-p^{b_3+1}\right)}.
   \end{eqnarray}

We also have
\begin{eqnarray}  \nonumber &&J^*_{\zeta,t}(\{a_1,a_2\},\{b_1,b_2\})
=P_2(a_1+b_1)P_2(a_2+b_2)+P_2(a_1+b_2)P_2(a_2+b_1)\\
&&\qquad -B_4(a_1,a_2,b_1,b_2) \\&&+ \nonumber
e^{-\ell(a_1+a_2+b_1+b_2)}A^*(a_1,a_2,b_1,b_2)
                \frac{Z_{\zeta}(\{a_1,a_2\},\{b_1,b_2\})Z_{\zeta}(\{-a_1,-a_2\},\{-b_1,-b_2\})}
{Z_{\zeta}^\dagger(\{a_1,a_2\},\{-a_1,-a_2\})Z_{\zeta}^\dagger
(\{b_1,b_2\},\{-b_1,-b_2\})}
 \\  \nonumber
                &&\qquad +W(a_1,b_1;a_2,b_2)+W(a_1,b_2;a_2,b_1)+W(a_2,b_1;a_1,b_2)+W(a_2,b_2;a_1,b_1)
\end{eqnarray}
where
\begin{eqnarray}&&
A^*(a_1,a_2,b_1,b_2)=\prod_p\frac{Z_p(\{a_1,a_2\},\{b_1,b_2\})Z_p(\{-a_1,-a_2\},\{-b_1,-b_2\})}
{Z_p(\{a_1,a_2\},\{-a_1,-a_2\})Z_p(\{b_1,b_2\},\{-b_1,-b_2\})}\\
&&  \nonumber \qquad \times p^{-a_1-a_2-b_1-b_2} \bigg( 1+\frac
{z_p(1-a_1-b_1)z_p(1-a_2-b_1)z_p(b_2-b_1)}{z_p(1)z_p(-a_1-b_1)z_p(-a_2-b_1)z_p(1+b_2-b_1)}   \\
&& \qquad \qquad  \nonumber +\frac {z_p(1-a_1-b_2)z_p(1-a_2-b_2)z_p(b_1-b_2)}
{ z_p(1)z_p(-a_1-b_2)z_p(-a_2-b_2)z_p(1+b_1-b_2)} \bigg).
\end{eqnarray}
and
\begin{eqnarray} &&
W(a_1,b_1;a_2,b_2)=P_1(a_1+b_1)
 \bigg\{P_2(a_2+b_2)-B_3(a_1,a_2;b_1,b_2)\\
&& \qquad \qquad  \nonumber + P_3(a_1,b_1,b_2)P_3(b_1,a_1,a_2)\bigg\}
\end{eqnarray}
with
\begin{eqnarray}&&
B_3(a_1,a_2;b_1,b_2)=\sum_p\log^2p \bigg(\frac{(p-1)^2
\left(p^{a_1+b_1}-1\right)^2
p^{a_1+b_1}}{\left(p^{a_1}-p^{a_2+1}\right) \left(-2
   p^{a_1+b_1}+p^{a_1+b_1+1}+1\right)^2 \left(p^{b_1}-p^{b_2+1}\right)}\\
   && \qquad \nonumber +\frac{C(a_1,a_2;b_1,b_2)}{\left(p^{a_1}-p^{a_2+1}\right) \left(-2
   p^{a_1+b_1}+p^{a_1+b_1+1}+1\right) \left(p^{b_2+1}-p^{b_1}\right)
   \left(p^{a_2+b_2+1}-1\right)}\\
   && \qquad \qquad \nonumber +\frac{1}{p^{a_2+b_2+1}-1} \bigg),
\end{eqnarray}
\begin{eqnarray} &&
C(a_1,a_2;b_1,b_2)=-p^{a_1+b_1}+2
   p^{a_1+b_1+1}-p^{a_2+b_1+2}-p^{2 a_1+2 b_1+1}+p^{a_1+a_2+2
   b_1+1}-p^{a_1+b_2+2}\nonumber \\
   &&\qquad +p^{a_2+b_2+2}+p^{2 a_1+b_1+b_2+1}-2
   p^{a_1+a_2+b_1+b_2+2}+p^{a_1+a_2+b_1+b_2+3};
   \end{eqnarray}
and
\begin{eqnarray} \nonumber
B_4(a_1,a_2;b_1,b_2)=\sum_p \frac{
 (3 - p^{1 + a_1 + b_1} - p^{1 + a_2 + b_1} -
                p^{1 + a_1 + b_2} - p^{1 + a_2 + b_2} +
                p^{2 + a_1 + a_2 + b_1 + b_2})\log^4 p}{(
                p^{1 + a_1 + b_1}-1)(
                p^{1 + a_2 + b_1})-1)(
                p^{1 + a_1 + b_2}-1)(
                p^{1 + a_2 + b_2}-1)}.
                \end{eqnarray}

\subsection{Auxiliary functions}\label{sect:auxfunc}
For ease of reference we list the various auxiliary functions we have introduced in this example section.

\begin{eqnarray}
\label{eq:AprimeA} A(x)&=&\prod_p
\frac{(1-\tfrac{1}{p^{1+x}})(1-\tfrac{2}{p}+\tfrac{1}{p^{1+x}})}
{(1-\tfrac{1}{p})^2},
\end{eqnarray}
 \begin{eqnarray}&&
A^*(a_1,a_2,b_1,b_2)=\prod_p\frac{Z_p(\{a_1,a_2\},\{b_1,b_2\})Z_p(\{-a_1,-a_2\},\{-b_1,-b_2\})}
{Z_p(\{a_1,a_2\},\{-a_1,-a_2\})Z_p(\{b_1,b_2\},\{-b_1,-b_2\})}\\
&&\qquad \nonumber \times p^{-a_1-a_2-b_1-b_2} \bigg( 1+\frac
{z_p(1-a_1-b_1)z_p(1-a_2-b_1)z_p(b_2-b_1)}{z_p(1)z_p(-a_1-b_1)z_p(-a_2-b_1)z_p(1+b_2-b_1)}   \\
&& \qquad \qquad \nonumber +\frac {z_p(1-a_1-b_2)z_p(1-a_2-b_2)z_p(b_1-b_2)}
{ z_p(1)z_p(-a_1-b_2)z_p(-a_2-b_2)z_p(1+b_1-b_2)} \bigg).
\end{eqnarray}
\begin{eqnarray}
\label{eq:BprimeA} B(x)&=&\sum_p\left(\frac{\log
p}{p^{1+x}-1}\right)^2,
\end{eqnarray}
\begin{eqnarray}
\label{eq:PprimeA} B_1(x,y)=  \sum_p\frac{\Big(1-\frac{1}{p^x}\Big)
\Big(1-\frac{1}{p^x}-\frac{1}{p^y} +\frac{1}{p^{1+y}}\Big) \log p}
{\Big(1-\frac{1}{p^{1-x+y}}\Big) \Big(1-\frac{1}{p^{1+y}}\Big)
\Big( 1-\frac{2}{p}+\frac{1}{p^{1+x}}\Big)p^{2-x+y}}.
\end{eqnarray}
\begin{eqnarray}
   B_2(a,b_1;b_2,b_3)=\sum_p \frac{(p-1) p^{2 b_1} \left(p^{a+b_1}-1\right) \left(p^{a+b_1}-p\right) \log ^2 p}{\left(-2 p^{a+b_1}+p^{a+b_1+1}+1\right)^2
   \left(p^{b_1}-p^{b_2+1}\right) \left(p^{b_1}-p^{b_3+1}\right)}.
   \end{eqnarray}
\begin{eqnarray}&&
B_3(a_1,a_2;b_1,b_2)=\sum_p\log^2p \bigg(\frac{(p-1)^2
\left(p^{a_1+b_1}-1\right)^2
p^{a_1+b_1}}{\left(p^{a_1}-p^{a_2+1}\right) \left(-2
   p^{a_1+b_1}+p^{a_1+b_1+1}+1\right)^2 \left(p^{b_1}-p^{b_2+1}\right)}\\
   && \qquad \nonumber +\frac{C(a_1,a_2;b_1,b_2)}{\left(p^{a_1}-p^{a_2+1}\right) \left(-2
   p^{a_1+b_1}+p^{a_1+b_1+1}+1\right) \left(p^{b_2+1}-p^{b_1}\right)
   \left(p^{a_2+b_2+1}-1\right)}\\
   && \qquad \qquad \nonumber +\frac{1}{p^{a_2+b_2+1}-1} \bigg)
\end{eqnarray}
\begin{eqnarray} \nonumber
B_4(a_1,a_2;b_1,b_2)=\sum_p \frac{
 (3 - p^{1 + a_1 + b_1} - p^{1 + a_2 + b_1} -
                p^{1 + a_1 + b_2} - p^{1 + a_2 + b_2} +
                p^{2 + a_1 + a_2 + b_1 + b_2})\log^4 p}{(
                p^{1 + a_1 + b_1}-1)(
                p^{1 + a_2 + b_1})-1)(
                p^{1 + a_1 + b_2}-1)(
                p^{1 + a_2 + b_2}-1)}.
                \end{eqnarray}
\begin{eqnarray} &&  \nonumber
C(a_1,a_2;b_1,b_2)=-p^{a_1+b_1}+2
   p^{a_1+b_1+1}-p^{a_2+b_1+2}-p^{2 a_1+2 b_1+1}+p^{a_1+a_2+2
   b_1+1}-p^{a_1+b_2+2}\\
   &&\qquad +p^{a_2+b_2+2}+p^{2 a_1+b_1+b_2+1}-2
   p^{a_1+a_2+b_1+b_2+2}+p^{a_1+a_2+b_1+b_2+3};
   \end{eqnarray}
\begin{eqnarray}
P_1(x)=e^{-\ell x}A(x){\zeta}(1+x){\zeta}(1-x)
\end{eqnarray}
\begin{eqnarray}
P_2(x)=\Big(\frac{{\zeta}'}{{\zeta}}\Big)'(1+x) - B(x),
\end{eqnarray}
\begin{eqnarray}
P_3(a,b,c)=B_1(a+b,a+c)+\frac{{\zeta}'}{{\zeta}}(1+a+c)-\frac{{\zeta}'}{{\zeta}}(1+c-b).
\end{eqnarray}
\begin{eqnarray}
\label{eq:QprimeA} Q(x,y)&=&-\sum_p \frac{\log^3 p}
{p^{2+x+y}(1-\frac{1}{p^{1+x}})( 1-\frac{1}{p^{1+y}})}
\end{eqnarray}
\begin{eqnarray} &&
W(a_1,b_1;a_2,b_2)=P_1(a_1+b_1)
 \bigg\{P_2(a_2+b_2)-B_3(a_1,a_2;b_1,b_2)\\
&& \qquad \qquad + P_3(a_1,b_1,b_2)P_3(b_1,a_1,a_2)\bigg\}\nonumber
\end{eqnarray}
\begin{eqnarray}
W_1(a,b_1;b_2,b_3)=P_1(a+b_1)(P_3(a,b_1,b_2)P_3(a,b_1,b_3)-B_2(a,b_1;b_2,b_3)).
\end{eqnarray}

\end{document}